\documentclass[12pt]{article}
\usepackage{amssymb}
\usepackage{amsfonts}
\usepackage{amsmath}
\usepackage{mathtools} 
\usepackage{color}
\usepackage{graphicx, float}
\usepackage{cite}
\usepackage{epsfig,mathrsfs}
\usepackage[bf,SL,BF]{subfigure}
\usepackage{color}
\usepackage{fancyhdr}
\newcommand{\proofend}{\hfill $\Box$ \vspace{2mm}}
\newtheorem{theorem}{Theorem}[section]
\newtheorem{assum}{Assumption}[section] 

\newtheorem{lemma}{Lemma}[section]
\newtheorem{definition}{Definition}[section]

\newcommand{\R}{\mathbb{R}}
\newcommand{\C}{\mathbb{C}}
\newcommand{\re}{\mbox{Re}}
\newcommand{\im}{\mbox{Im}}

\newcommand{\Lk}{{\bf L}(k)}
\newcommand{\tLk}{\tilde{{\bf L}}(k)}

\newcommand{\vecE}{{\bf E}}
\newcommand{\vecL}{{\bf L}}
\newcommand{\vecH}{{\bf H}}
\newcommand{\vecu}{{\bf u}}
\newcommand{\vecw}{  {\bf w} }
\newcommand{\vecv}{  {\bf v}  }
\newcommand{\vecM}{{\bf M}}
\newcommand{\vectM}{\tilde{\bf M}}
\newcommand{\vecW}{{\bf W}}
\newcommand{\vecU}{{\bf U}}
\newcommand{\vecV}{{\bf V}}
\newcommand{\vecQ}{{\bf Q}}
\newcommand{\vecP}{{\bf P}}

\newcommand{\vecJ}{{\bf J}}

\newcommand{\x}{{\bf x}}
\newcommand{\Tka}{{\bf T}_k}
\newcommand{\Tkb}{{\bf T}_{k_1}}
\newcommand{\Tkw}{{\bf T}_{k\sqrt{n}}}
\newcommand{\Kkw}{{\bf K}_{k\sqrt{n}}}

\newcommand{\Tkc}{{\bf T}_{k_2}}
\newcommand{\Kka}{{\bf K}_k}
\newcommand{\Kkb}{{\bf K}_{k_1}}
\newcommand{\Kkc}{{\bf K}_{k_2}}
\newcommand{\Ska}{{\bf S}_k}
\newcommand{\Skb}{{\bf S}_{k_1}}

\newcommand{\mdiv}{\mbox{div\,}}
\newcommand{\mcurl}{\mbox{curl\,}}

\graphicspath{{pics/}}

\begin{document}
 \pagenumbering{arabic}
\title{Boundary Integral Equations for the Transmission Eigenvalue Problem for Maxwell's Equations}

\author{Fioralba Cakoni\footnote{Department of Mathematical Sciences,
University of Delaware, Newark, Delaware 19716, USA,
 (fcakoni@udel.edu)}
 \ Houssem Haddar\footnote{CMAP, Ecole Polytechnique, Route de Saclay, 91128 Palaiseau Cedex, France, (Houssem.Haddar@inria.fr)}
 \  and Shixu Meng\footnote{Department of Mathematical Sciences,
University of Delaware, Newark, Delaware 19716, USA,
\;(sxmeng@udel.edu)}}
\date{}
\maketitle

\begin{abstract}
In this paper we consider the transmission eigenvalue problem  for Maxwell's equations corresponding to non-magnetic inhomogeneities with contrast in electric permittivity that changes sign inside its support. We formulate the transmission eigenvalue problem as an equivalent homogeneous  system of boundary integral equation, and  assuming that the contrast  is constant near the boundary of the support of the inhomogeneity, we prove that  the  operator associated with this system is Fredholm of  index zero and depends analytically on the wave number. Then we show the existence of  wave numbers that are not transmission eigenvalues which by an application of the analytic Fredholm theory implies that the set of transmission eigenvalues is discrete with positive infinity as the only accumulation point.  
\end{abstract} 


\section{Introduction}  
The transmission eigenvalue problem is  related to the scattering problem for an inhomogeneous media. In the current paper the underlying scattering problem is the scattering of electromagnetic waves by a (possibly anisotropic) non-magnetic material of bounded support $D$ situated in homogenous background, which in terms of the electric field reads: 
\begin{eqnarray}
&\mcurl \mcurl \vecE^s-k^2 \vecE^s = 0 \quad &\mbox{in} \quad {\mathbb R}^3\setminus\overline{D} \label{maxwellg} \\
&\mcurl \mcurl \vecE-k^2 N \vecE = 0 \quad &\mbox{in} \quad D  \\
&\nu \times \vecE = \nu \times \vecE^s+nu \times \vecE^i \quad &\mbox{on} \quad \partial D\\ 
&\nu \times \mcurl \vecE = \nu \times \mcurl \vecE^s+\nu \times \mcurl \vecE^i \quad &\mbox{on} \quad \partial D \\
&\lim\limits_{r\to \infty}\left(\mcurl \vecE^s\times x-ikr\vecE^s\right)=0 \label{sm} &
\end{eqnarray}
where $\vecE^i$ is the incident electric field, $\vecE^s$ is the scattered electric field and $N(x)=\displaystyle{\frac{\epsilon(x)}{\epsilon_0}+i\frac{\sigma(x)}{\omega\epsilon_0}}$ is the matrix index of refraction, $k=\omega\sqrt{\epsilon_0\mu_0}$ is the wave number corresponding to the background and the  frequency $\omega$ and the Silver-M{\"u}ller radiation condition is satisfied uniformly with respect to $\hat x=x/r$, $r=|x|$.  The difference $N-I$, in the following,  is refereed to as the contrast in the media. In scattering theory, transmission eigenvalues can be seen as the extension of the notion of resonant frequencies for impenetrable objects to the case of penetrable media. The transmission eigenvalue problem is related to non-scattering incident fields. Indeed, if $\vecE^i$ is such that $\vecE^s=0$ then $\vecE|_{D}$ and $\vecE_0=\vecE^i|_{D}$ satisfy the following homogenous problem
\begin{eqnarray}
&\mcurl \mcurl \vecE-k^2 N \vecE = 0 \quad &\mbox{in} \quad D  \label{maxwell1gg} \\
&\mcurl \mcurl \vecE_0-k^2 \vecE_0 = 0 \quad &\mbox{in} \quad D \label{maxwellgg} \\
&\nu \times \vecE = \nu \times \vecE_0 \quad &\mbox{on} \quad \partial D \label{dirichletgg} \\ 
&\nu \times \mcurl \vecE = \nu \times \mcurl \vecE_0 \quad &\mbox{on} \quad \partial D  \label{neumanngg}
\end{eqnarray}
which is referred to as the transmission eigenvalue problem. Conversely, if (\ref{maxwell1gg})-(\ref{neumanngg})  has a nontrivial solution $\vecE$ and $\vecE_0$ and $\vecE_0$ can be extended outside $D$ as a solution to $\mcurl \mcurl \vecE_0-k^2 \vecE_0 = 0$, then if this  extended  $\vecE_0$  is considered as the incident field the corresponding scattered field is $\vecE^s=0$. 

The transmission eigenvalue problem is a  nonlinear and non-selfadjoint eigenvalue problem that  is not covered by the standard theory of eigenvalue problems for  elliptic equations. For a long time research on the transmission eigenvalue problem mainly focussed on showing that transmission eigenvalues form at most a discrete set and we refer the reader to the survey paper \cite{CakHad2} for the state of the art  on this question up to 2010.  From a  practical point of view the question of discreteness was  important to answer,  since  sampling methods for reconstructing the support of an inhomogeneous medium \cite{Cak1}, \cite{8} fail if the interrogating frequency corresponds to a transmission eigenvalue. On the other hand, due to the non-selfadjointness of the transmission eigenvalue problem, the existence of transmission eigenvalues for non-spherically stratified media remained open for more than 20 years until Sylvester and P\"aiv\"arinta \cite{PS} showed  the existence of at least one transmission eigenvalue  provided that the contrast in the medium is large enough. A full answer on the existence of transmission eigenvalues was given by Cakoni, Gintides and Haddar \cite{1} where  the existence of an infinite set of transmission eigenvalue was proven only  under  the assumption that the contrast in the medium does not change sign and is bounded away from zero (see also \cite{9} \cite{11-1}, \cite{cos-had} and  \cite{kirsch} for Maxwell's equation). Since the appearance of these papers there has been an explosion of interest in the transmission eigenvalue problem and the papers in the Special Issue of Inverse Problems on  Transmission Eigenvalues, Volume 29, Number 10, October 2013,   are representative of the myriad directions that this research has taken.  

The discreteness and existence of transmission eigenvalues is very well understood under the assumption that the contrast does not change sign in all of $D$. Recently, for the scalar Helmholtz type equation, several papers have appeared that address both the question of discreteness and existence of transmission eigenvalue assuming that the  contrast is  of one sign only in a neighborhood of the inhomogeneity's  boundary $\partial D$, \cite{Bon-Che-Had-2011},  \cite{exterior}, \cite{ConH}, \cite{LV3}, \cite{LV4}, \cite{robbiano} and \cite{syld}. The picture is not the same for the transmission eigenvalue problem for the Maxwell's equation. The only result in this direction is the proof of discreteness of transmission eigenvalues in \cite{luca} for magnetic materials, i.e. when there is contrast in both the electric prematurity and magnetic permeability. The $T$-coercivity approach used in  \cite{luca} does not apply to our problem (\ref{maxwell1gg})-(\ref{neumanngg}), which mathematically has a different structure form the case of magnetic materials and this paper is dedicated to study the discreteness of transmission eigenvalues for the considered problem under weaker assumptions of $N-I$. Before specifying  our assumptions and approach let us rigorously  formulate our  transmission eigenvalue problem. 

{\bf Formulation of the Problem}: Let $D\in {\mathbb R}^3$ be a bounded open and connected region with $C^2$-smooth boundary $\partial D:=\Gamma$ (we call it $\Gamma$ for convenience of notation as will be seen later) and let $\nu$ denotes the outward unit normal vector on $\Gamma$. In general we consider a $3\times3$  matrix-valued function $N$ with $L^\infty(D)$ entries such that $\overline{\xi} \cdot \re(N)\xi\geq \alpha>0$ and $\overline{\xi} \cdot \im(N)\xi\geq 0$ in $D$ for every $\xi\in {\mathbb C}^3$, $|\xi|=1$. The transmission eigenvalue problem can be formulated as finding $\vecE, \vecE_0 \in \vecL^2(D)$,  $\vecE-\vecE_0 \in \vecH_0(\mbox{curl}^2, D)$ that satisfy 
\begin{eqnarray}
&\mcurl \mcurl \vecE-k^2 N \vecE = 0 \quad &\mbox{in} \quad D  \label{maxwell1g} \\
&\mcurl \mcurl \vecE_0-k^2 \vecE_0 = 0 \quad &\mbox{in} \quad D \label{maxwellg} \\
&\nu \times \vecE = \nu \times \vecE_0 \quad &\mbox{on} \quad \Gamma \label{dirichletg} \\ 
&\nu \times \mcurl \vecE = \nu \times \mcurl \vecE_0 \quad &\mbox{on} \quad \Gamma  \label{neumanng}
\end{eqnarray}
where
$$
\vecL^2(D):= \left\{\vecu: \vecu_j \in L^2(D), j=1,2,3 \right\},
$$
$$\vecH(\mcurl^2, D):=\left\{\vecu: \vecu\in \vecL^2(D), \mcurl \vecu\in \vecL^2(D) \; \mbox{and} \; \mcurl \mcurl \vecu \in \vecL^2(D)\right\},
$$ 
$$
\vecH_0(\mcurl^2, D):=\left\{\vecu: \vecu\in \vecH(\mcurl^2, D), \gamma_t\vecu=0 \; \mbox{and}\; \gamma_t \mcurl \vecu=0 \; \mbox{on} \; \Gamma \right\}.
$$
\begin{definition}
Values of $k\in{\mathbb C}$ for which the (\ref{maxwell1g})-(\ref{neumanng}) has a nontrivial solution $\vecE, \vecE_0 \in \vecL^2(D)$,  $\vecE-\vecE_0 \in \vecH_0(\mbox{\em curl}^2, D)$ are called transmission eigenvalues.
\end{definition}
It is well-known \cite{1}, \cite{haddar}  that, if  $\re(N-I)$ has one sign in $D$ the transmission eigenvalues form at most a discrete set with $+\infty$ as the only possible accumulation point, and if in addition $\im(N)=0$, there exists an infinite set of real transmission eigenvalues. Our main concern is to understand the structure of the transmission eigenvalue problem in the case when $\re(N-I)$ changes sign inside $D$. More specifically in this case we show that the transmission eigenvalues form at most a discrete set using an equivalent integral equation formulation of the transmission eigenvalue problem following the boundary integral equations approach developed in \cite{ConH}. The assumption on the real part of the contract $N-I$  that we need in our analysis  will become more precise later in the paper, but roughly speaking in our approach we allow for  $\re(N-I)$ to change sign in a compact subset of $D$.  To this end,  in the next section we  consider the simplest case when the electric permittivity is constant, i.e. $N=nI$ with positive $n\neq 1$, for which we develop and analyze an equivalent system of integral equations formulation of the corresponding transmission eigenvalue problem. This system of integral equations will then be a building block to study the more general case of the electric permittivity $N$. We note that the extension  to Maxwell's equations of the approach in \cite{ConH} is not a trivial task due to the more peculiar mapping properties of  the electromagnetic boundary integral operators as it will become clear in the paper.
\section{Boundary Integral Equations  for Constant Electric Permittivity}
Let $n>0$ be a constant such that $n\neq 1$ and consider the problem of  finding $\vecE, \vecE_0 \in \vecL^2(D)$,  $\vecE-\vecE_0 \in \vecH_0(\mbox{curl}^2, D)$ that satisfy 
\begin{eqnarray}
\mcurl \mcurl \vecE-k^2 n \vecE = 0 \quad &\mbox{in}& \quad D  \label{maxwell1} \\
\mcurl \mcurl \vecE_0-k^2 \vecE_0 = 0 \quad &\mbox{in}& \quad D \label{maxwell} \\
\nu \times \vecE = \nu \times \vecE_0 \quad &\mbox{on}& \quad \Gamma \label{dirichlet} \\ 
\nu \times (\mcurl \vecE) = \nu \times (\mcurl \vecE_0) \quad &\mbox{on}& \quad \Gamma  \label{neumann}
\end{eqnarray}
In the following we set $k_1:=k\sqrt{n}$. Before formulating the transmission eigenvalue problem as an equivalent system of boundary  integral equations, we recall  several integral operators and study their mapping properties. To this end, let us define the Hilbert  spaces of tangential fields defined on $\Gamma$:
\begin{eqnarray}
\vecH^{s_1,s_2}(\mbox{div}, \Gamma) :=\{ \vecu \in \vecH_t^{s_1}(\Gamma), \mdiv_{\Gamma}  \vecu \in H^{s_2}(\Gamma) \}, \nonumber \\
\vecH^{s_1,s_2}(\mbox{curl}, \Gamma) :=\{ \vecu \in \vecH_t^{s_1}(\Gamma), \mcurl_{\Gamma}  \vecu \in \vecH^{s_2}(\Gamma) \} \nonumber  
\end{eqnarray} 
endowed with the respective natural norms, where $\mcurl_{\Gamma}$ and $\mdiv_{\Gamma}$ are the surface curl and divergence operator, respectively, and for later use $\nabla_{\Gamma}$ denotes the tangential gradient operator. (Note that the boldface indicate  vector spaces of vector fields, whereas non-bold face indicate vector spaces of scalar fields.) If  $\gamma_\Gamma \,\vecu = \nu \times (\vecu \times \nu)$ denotes the tangential trace of a vector field $\vecu$ on the boundary $\Gamma$, we  define the boundary integral operators:
\begin{eqnarray}\label{dT}
\Tka(\vecu):= \frac{1}{k}\gamma_\Gamma \left(k^2 \int_\Gamma \Phi_k(\cdot,{{\bf y}}) \vecu({{\bf y}})\,ds_y + \nabla_\Gamma \int_\Gamma \Phi_k(\cdot, {{\bf y}}) \mdiv_\Gamma \vecu({{\bf y}}) \, ds_y \right),
\end{eqnarray} 
and
\begin{eqnarray}\label{dK}
\Kka(\vecu):= \gamma_\Gamma\left(\mcurl \int_\Gamma \Phi_k(\cdot,y) \vecu({\bf y})\,ds_y\right)
\end{eqnarray} 
where 
$$\Phi_k(x,y)=\frac{1}{4\pi}\frac{e^{ik|x-y|}}{|x-y|}$$
is the fundament solution of the Helmholtz equation $\Delta u+k^2u=0$. Referring to \cite{ConH} and \cite{mclean}  for the mapping properties of the single layer potential 
\begin{equation}\label{pics}
S_k(\varphi):= \int_{\Gamma}\Phi_k(\cdot, {\bf y}) \varphi({\bf y}) ds_y,
\end{equation}
 with scalar densities $\varphi$, we have that  the boundary integral operator 
\begin{equation}\label{pic}
 \Ska (\vecu) = \int_{\Gamma} \Phi_k(\cdot, {\bf y}) \vecu({\bf y}) \,ds
\end{equation}
acting on vector fields $\vecu$  is bounded from $\vecH^{-\frac{1}{2}+s}(\Gamma)$ to $\vecH^{\frac{1}{2}+s}(\Gamma) $ for $-1\leq s\leq 1$ and  hence  
$$
\Tka: \vecH^{-\frac{1}{2},-\frac{3}{2}}(\mbox{div}, \Gamma) \to \vecH^{-\frac{1}{2},-\frac{3}{2}}(\mbox{curl}, \Gamma) 
$$
$$
\Kka: \vecH^{-\frac{3}{2},-\frac{1}{2}}(\mbox{div}, \Gamma) \to \vecH^{-\frac{3}{2},-\frac{1}{2}}(\mbox{curl}, \Gamma) 
$$
are bounded linear operators. Now from the Stratton-Chu formula \cite{coltonkress} we have that
\begin{eqnarray}
\vecE_0(\x) &=& \mcurl \int_{\Gamma} (\vecE_0 \times \nu)({\bf y}) \Phi_k({\bf x},{\bf y}) ds_y +\int_\Gamma (\mcurl\vecE_0 \times \nu)({\bf y}) \Phi_k({\bf x},{\bf y})ds_y \nonumber \\ 
&+&  \frac{1}{k^2}\nabla \int_\Gamma \mdiv_\Gamma(\mcurl\vecE_0 \times \nu)({\bf y}) \Phi_k({\bf x}, {\bf y}) ds_y  \qquad \mbox{for}\quad \x\in D \nonumber
\end{eqnarray}
with similar expression for $\vecE$ where $k$  is replaced by $k_1:=k\sqrt{n}$ and hence we have the integral expression for $\vecE-\vecE_0$. Note by taking the difference $\vecE-\vecE_0$ we have the corresponding kernel $\Phi_{k_1}({\bf x},{\bf y})-\Phi_k({\bf x},{\bf y})$ is a smooth function of $\bf{x},\bf{y}$, and approaching the boundary $\Gamma$ and noting $\vecE \times \nu=\vecE_0 \times \nu$ and $\mcurl\vecE \times \nu=\mcurl\vecE_0 \times \nu$ we have
\begin{eqnarray}
\;\;\;\gamma_\Gamma (\vecE-\vecE_0) &= & (\Kka-\Kkb)(\vecE_0 \times \nu )  + \left(\frac{1}{k}\Tka-\frac{1}{k_1}\Tkb\right) (\mcurl\vecE_0 \times \nu), \nonumber \\
\gamma_\Gamma \mcurl (\vecE-\vecE_0) &= & (\Kka-\Kkb)(\mcurl \vecE_0 \times \nu )  +\left(k\Tka-k_1\Tkb\right) (\vecE_0 \times \nu). \nonumber 
\end{eqnarray}
From the boundary conditions (\ref{dirichlet}) and (\ref{neumann}) we have $\gamma_\Gamma (\vecE-\vecE_0)=0$ and $\gamma_\Gamma \mcurl (\vecE-\vecE_0)=0$, i.e.
\begin{eqnarray}
\Kka (\vecE_0 \times \nu) +\frac{1}{k} \Tka(\mcurl\vecE_0 \times \nu) &-& \Kkb (\vecE \times \nu) -\frac{1}{k_1} \Tkb(\mcurl\vecE \times \nu)=0, \label{ik1} \\
\Kka (\mcurl \vecE_0 \times \nu) +k\Tka (\vecE_0 \times \nu) &-& \Kkb (\mcurl \vecE \times \nu) -k_1\Tkb (\vecE \times \nu)=0. \label{ik2}
\end{eqnarray}
Introducing  $\vecM=\vecE \times \nu=\vecE_0 \times \nu$ and $\vecJ=\mcurl\vecE \times \nu=\mcurl\vecE_0 \times \nu$, we arrive at the following homogeneous system of boundary integral equations
\begin{eqnarray} \label{int}
\left( \begin{array}{cc}
k_1\Tkb-k\Tka & \Kkb-\Kka\\
\Kkb-\Kka & \frac{1}{k_1}\Tkb-\frac{1}{k}\Tka \end{array} \right) 
\left( \begin{array}{c} \vecM \\ \vecJ \end{array} \right) =\left( \begin{array}{c} 0\\ 0 \end{array} \right)
\end{eqnarray}
 for the unknowns $\vecM$ and $\vecJ$. Let us define
\begin{eqnarray}
\Lk=: 
\left( \begin{array}{cc}
k_1\Tkb-k\Tka & \Kkb-\Kka\\
\Kkb-\Kka & \frac{1}{k_1}\Tkb-\frac{1}{k}\Tka \end{array} \right)=\left( \begin{array}{cc}
k\sqrt{n}\Tkw-k\Tka & \Kkw-\Kka\\
\Kkw-\Kka & \frac{1}{k\sqrt{n}}\Tkw-\frac{1}{k}\Tka \end{array} \right). \label{lk}
\end{eqnarray}
Note that while the operator $ \Kkb-\Kka$  is a  smoothing pseudo-differential operator of order 2 (see e.g.  \cite{ConH} and \cite{hsiao}), the operators in the  main diagonal have a mixed structure. Indeed,  from the expressions
\begin{eqnarray}
k_1\Tkb-k\Tka &=& ({k^2_1}{\mathbf S}_{{k_1}}-k^2{\mathbf S}_{k})+ \nabla_{\Gamma}\circ\left({S}_{{k_1}}-{S}_{k}\right) \circ \mbox{div}_{\Gamma}\label{prop}\\
\frac{1}{k_1}\Tkb-\frac{1}{k}\Tka &=& \left({\mathbf S}_{{k_1}}-{\mathbf S}_{k}\right)+\nabla_{\Gamma}\circ \left(\frac{1}{{k^2_1}}{S}_{{k_1}}-  \frac{1}{{k}^2} {S}_{k}\right)\circ \mbox{div}_{\Gamma}\nonumber
\end{eqnarray}
where $S$ and ${\mathbf S}$ are defined by (\ref{pics}) and (\ref{pic}) respectively,  we can see that these operators have different behavior component-wise. Hence a more delicate analysis is called for to find the correct function spaces for $\vecM, \vecJ$ and their dual spaces in order to analyze the  mapping properties of the operator $\Lk$. 
\begin{lemma} \label{dual}
The dual space of $\,\vecH^{-\frac{3}{2},-\frac{1}{2}}(\mbox{\em div}, \Gamma)$ is $\vecH^{-\frac{1}{2},\frac{1}{2}}(\mbox{\em curl}, \Gamma)$. For $\vecu^t \in \vecH^{-\frac{1}{2},\frac{1}{2}}(\mbox{\em curl}, \Gamma)$ and $\vecu \in \vecH^{-\frac{3}{2},-\frac{1}{2}}(\mbox{\em div}, \Gamma) $, $\left<\vecu^t, \vecu\right>$ is understood by duality with respect to ${\bf L}^2(\Gamma)$ as a pivot space.
\end{lemma}
\begin{proof}
For any tangential fields  $\vecu \in \vecH^{-\frac{3}{2},-\frac{1}{2}}(\mbox{div}, \Gamma) $ and $\vecu^t \in \vecH^{-\frac{1}{2},\frac{1}{2}}(\mcurl, \Gamma)$, we consider the corresponding Helmholtz orthogonal decomposition
\begin{eqnarray*}
\vecu = \overrightarrow{\mcurl}_{\Gamma} q + \nabla_{\Gamma} p, \quad\vecu^t = \overrightarrow{\mcurl}_{\Gamma} q^t + \nabla_{\Gamma} p^t.
\end{eqnarray*}
Since $\mdiv_{\Gamma} \vecu = \mdiv_{\Gamma}\nabla_{\Gamma} p=\Delta_{\Gamma} p \in   H^{-\frac{1}{2}}(\Gamma)$  we have by eigensystem expansion (e.g. \cite{nedelec}) that $\nabla_\Gamma p \in \vecH^{\frac{1}{2}}(\Gamma)$. Similarly, from the fact that $ \mcurl_{\Gamma}  \vecu^t \in \vecH^{\frac{1}{2}}(\Gamma)$ we obtain that $\overrightarrow{\mcurl}_{\Gamma} q^t \in \vecH^{\frac{3}{2}}(\Gamma)$. Now 
\begin{eqnarray}
\left<\vecu^t, \vecu\right> &=& \left<\overrightarrow{\mcurl}_{\Gamma} q^t + \nabla_{\Gamma} p^t,\overrightarrow{\mcurl}_{\Gamma} q + \nabla_{\Gamma} p\right> \nonumber \\
&=&\left <\overrightarrow{\mcurl}_{\Gamma} q^t,\overrightarrow{\mcurl}_{\Gamma} q\right> +\left <\nabla_{\Gamma} p,\nabla_{\Gamma} p^t\right>. \nonumber 
\end{eqnarray}
Hence the right hand side is well defined in the sense of duality of $\vecH^{\frac{3}{2}}(\Gamma)$-$\vecH^{-\frac{3}{2}}(\Gamma)$ and $\vecH^{\frac{1}{2}}(\Gamma)$-$\vecH^{-\frac{1}{2}}(\Gamma)$, and thus $\vecH^{-\frac{1}{2},\frac{1}{2}}(\mbox{curl}, \Gamma)$ is in the dual space of $\,\vecH^{-\frac{3}{2},-\frac{1}{2}}(\mbox{div}, \Gamma)$. 

Furthermore, if $\vecu^t = \overrightarrow{\mcurl}_{\Gamma} q^t + \nabla_{\Gamma} p^t$ is in the dual space of $\,\vecH^{-\frac{3}{2},-\frac{1}{2}}(\mbox{div}, \Gamma)$, then $\left<\vecu^t, \cdot \right>$ is continuous and linear on $\,\vecH^{-\frac{3}{2},-\frac{1}{2}}(\mbox{div}, \Gamma)$. Then for $\vecu = \overrightarrow{\mcurl}_{\Gamma} q$
$$
\left<\vecu^t,\vecu \right>=\left<\overrightarrow{\mcurl}_{\Gamma} q^t,\overrightarrow{\mcurl}_{\Gamma} q \right>.
$$
Notice $\overrightarrow{\mcurl}_{\Gamma} q$ is only in $\vecH^{-\frac{3}{2}}(\Gamma)$, therefore  by eigensystem analysis $\overrightarrow{\mcurl}_{\Gamma} q^t \in \vecH^{\frac{3}{2}}(\Gamma)$ and $\mcurl_\Gamma \overrightarrow{\mcurl}_{\Gamma} q^t \in H^{\frac{1}{2}}(\Gamma)$, i.e. $\mcurl_\Gamma \vecu^t \in H^{\frac{1}{2}}(\Gamma)$. Now for $\vecu = \nabla_{\Gamma} p$ where $\nabla_\Gamma p \in \vecH^{\frac{1}{2}}(\Gamma)$
$$
\left<\vecu^t,\vecu \right>=\left<\nabla_{\Gamma} p^t, \nabla_{\Gamma} p \right>.
$$
Then $\nabla_{\Gamma} p^t \in \vecH^{-\frac{1}{2}}(\Gamma)$. Therefore $\vecu^t \in \vecH^{-\frac{1}{2},\frac{1}{2}}(\mbox{curl}, \Gamma)$. Now we have proved the lemma.
\end{proof} \proofend

In the following the spaces $\vecH^{-\frac{3}{2},-\frac{1}{2}}(\mbox{div}, \Gamma) $ and $ \vecH^{-\frac{1}{2},\frac{1}{2}}(\mcurl, \Gamma)$ are considered dual to each other in the duality  defined in Lemma \ref{dual}.
In the next lemma  we establish some mapping properties of the operator $\Lk$ given by (\ref{lk}). 
\begin{lemma} \label{bddLk}
For a fixed $k$, the linear operator 
$$
\Lk: \vecH_t^{-\frac{1}{2}}(\Gamma) \times \vecH^{-\frac{3}{2},-\frac{1}{2}}(\mbox{\em div}, \Gamma)  \to  \vecH_t^{\frac{1}{2}}(\Gamma) \times\vecH^{-\frac{1}{2},\frac{1}{2}}(\mbox{\em curl}, \Gamma)
$$
is bounded. Moreover, the family of operators $\Lk$ depends analytically on $k \in \C \backslash \R_{-}$.
\end{lemma}
\begin{proof}
Let $\vecE, \vecE_0 \in \vecL^2(D)$,  $\vecE-\vecE_0 \in \vecH_0(\mbox{curl}^2, D)$ be a solution to the transmission eigenvalue problem (\ref{maxwell1})-(\ref{neumann}). Hence 
$$
\vecM = \vecE \times \nu \in \vecH_t^{-\frac{1}{2}}(\Gamma), \quad \vecJ = \mcurl \vecE \times \nu \in \vecH_t^{-\frac{3}{2}}(\Gamma).
$$
Noting that $\mdiv_{\Gamma}(\mcurl\vecE \times \nu)=\mcurl_{\Gamma} \mcurl \vecE=\mcurl^2\vecE\cdot \nu|_{\Gamma}$, we have that  $\mdiv_{\Gamma}\vecJ \in \vecH_t^{-\frac{1}{2}}(\Gamma)$ and  therefore $(\vecM, \vecJ) \in  \vecH_t^{-\frac{1}{2}}(\Gamma) \times \vecH^{-\frac{3}{2},-\frac{1}{2}}(\mbox{div}, \Gamma) $. It is known from \cite{ConH} that $\Ska$, $\Skb-\Ska$, $\Kkb-\Kka$ are smoothing operators of order $1$, $3$ and $2$ respectively. Then using (\ref{prop}) we have that the following operators are bounded
\begin{eqnarray*}
&k_1\Tkb-k\Tka &: \qquad \vecH_t^{-\frac{1}{2}}(\Gamma) \to \vecH_t^{\frac{1}{2}}(\Gamma) \\
&\;\;\Kkb-\Kka &: \qquad \vecH_t^{-\frac{3}{2}}(\Gamma) \to \vecH_t^{\frac{1}{2}}(\Gamma) \\
&\frac{1}{k_1}\Tkb-\frac{1}{k}\Tka &: \qquad \vecH^{-\frac{3}{2},-\frac{1}{2}}(\mbox{div}, \Gamma)  \to \vecH_t^{-\frac{1}{2}}(\Gamma)
\end{eqnarray*}
 Moreover 
\begin{eqnarray*}
&&\mcurl_{\Gamma} \left( (\Kkb-\Kka )\vecM+ (\frac{1}{k_1}\Tkb-\frac{1}{k}\Tka)\vecJ\right)\\
= &&\mcurl_{\Gamma}( \Kkb-\Kka )\vecM + \mcurl_{\Gamma}(\Skb-\Ska)\vecJ \in \vecH_t^{\frac{1}{2}}(\Gamma),
\end{eqnarray*}
and hence
\begin{eqnarray*}
( k_1\Tkb-k\Tka )\vecM + (\Kkb-\Kka)\vecJ &\in& \vecH_t^{\frac{1}{2}}(\Gamma),\\
( \Kkb-\Kka )\vecM + \left(\frac{1}{k_1}\Tkb-\frac{1}{k}\Tka\right)\vecJ &\in&\vecH^{-\frac{1}{2},\frac{1}{2}}(\mbox{curl}, \Gamma), 
\end{eqnarray*}
Hence $\Lk$ is bounded. Note that since every component of $\Lk$ is analytic on $\C \backslash \R_{-}$, then $\Lk$ is analytic on $\C \backslash \R_{-}$ (recall that $k_1=k\sqrt{n}$).
\end{proof} \proofend

We need the following lemma to show the equivalence between the transmission eigenvalue problem and the system of integral equations (\ref{int}).
\begin{lemma} \label{jumprelations}
Let $\Omega$ be any bounded open region in $\R^3$ and denote $\vecV(\mbox{\em curl}^2, \Omega):= \{ \vecu: \vecu \in \vecL^2(\Omega), \mbox{\em curl}^2 \vecu \in \vecL^2(\Omega) \}$. For ${\bf \varphi} \in \vecH_t^{-\frac{1}{2}}(\Gamma)$, ${\bf \psi} \in \vecH^{-\frac{3}{2},-\frac{1}{2}}(\mbox{\em div}, \Gamma)$, we define 
$$
\vectM_1({\bf \varphi})({\bf x}):= \mbox{\em curl} \int_\Gamma \Phi_k({\bf x},{\bf y}){\bf \varphi}({\bf y})  ds_y,\quad x \in \R^3 \backslash \Gamma, 
$$
and 
$$
\vectM_2({\bf \psi})({\bf y}):= \int_\Gamma \Phi_k({\bf x},{\bf y}) {\bf \psi}({\bf y})  ds_y, \quad x \in \R^3 \backslash \Gamma.
$$ 
Then $\vectM_1$ is continuous from $\vecH_t^{-\frac{1}{2}}(\Gamma)$ to $\vecV(\mbox{\em curl}^2, D^{\pm})$ and $\vectM_2$ is continuous from $\vecH^{-\frac{3}{2},-\frac{1}{2}}(\mbox{\em div}, \Gamma)$ to $\vecV(\mbox{\em curl}^2, D^{\pm})$ where $D^-=D$ and $D^+=B_R \backslash \overline{D}$ with a sufficient large ball $B_R$ containing the closure of $D$. Furthermore the following jump relations hold
\begin{eqnarray}
[\gamma_{t} \vectM_1({\bf \varphi}) ] = {\bf \varphi} \quad \mbox{in} \quad \vecH_t^{-\frac{1}{2}}(\Gamma), \label{jumprelations1}\\ \,
[\gamma_{t} \mbox{\em curl} \vectM_1({\bf \varphi}) ] = 0 \quad \mbox{in} \quad \vecH_t^{-\frac{3}{2}}(\Gamma),\label{jumprelations2} \\ \,
[\gamma_{t} \mbox{\em curl} \vectM_2({\bf \psi}) ] = {\bf \psi} \quad \mbox{in} \quad \vecH_t^{-\frac{3}{2}}(\Gamma),\label{jumprelations3} \\ \,
[\mbox{\em div}_{\Gamma} \gamma_t \mbox{\em curl} \vectM_2({\bf \psi}) ] = \mbox{\em div}_{\Gamma} {\bf \psi} \quad \mbox{in} \quad H^{-\frac{1}{2}}(\Gamma).
\label{jumprelations4}
\end{eqnarray}
\end{lemma}
\begin{proof}
Let us denote by $<\cdot,\cdot>$ the $\vecH_t^{\frac{1}{2}}(\Gamma)$-$\vecH_t^{-\frac{1}{2}}(\Gamma)$ or $H^{\frac{1}{2}}(\Gamma)$-$H^{-\frac{1}{2}}(\Gamma)$ duality product.
Since ${\bf \varphi} \in \vecH_t^{-\frac{1}{2}}(\Gamma)$, then from the classical results for single layer potentials
$$
\| \vectM_1({\bf \varphi}) \|_{\vecL^2(D^{\pm})} \le c\left\lVert \int_\Gamma \Phi_k({\bf x},{\bf y}){\bf \varphi}({\bf y})  ds_y \right\rVert_{\vecH^1(D^{\pm})} \le c\|{\bf \varphi}\|_{\vecH_t^{-\frac{1}{2}}(\Gamma)}\quad
$$
and since $\mcurl^2 \vectM_1({\bf \varphi})-k^2 \vectM_1({\bf \varphi})=0$ in $D^\pm$, then
$$
\| \mcurl^2 \vectM_1({\bf \varphi}) \|_{\vecL^2(D^{\pm})} =|k^2| \| \vectM_1({\bf \varphi}) \|_{\vecL^2(D^{\pm})} \le c\|{\bf \varphi}\|_{\vecH_t^{-\frac{1}{2}}(\Gamma)} \quad
$$
where $c$ is some constant depending on $k$.
For ${\bf \psi} \in \vecH^{-\frac{3}{2},-\frac{1}{2}}(\mbox{div}, \Gamma)$, we have from \cite{ConH}
$$
\| \vectM_2({\bf \psi}) \|_{\vecL^2(D^{\pm})} \le c\|{\bf \psi}\|_{\vecH_t^{-\frac{3}{2}}(\Gamma)} \, .
$$
Notice that 
$$
\mcurl^2 \vectM_2({\bf \psi})({\bf x}) =k^2 \int_\Gamma \Phi_k({\bf x},{\bf y}) {\bf \psi}({\bf y})  ds_y + \nabla \int_\Gamma \mdiv_\Gamma  {\bf \psi}({\bf y}) \Phi_k(\cdot, {\bf y}) ds_y 
$$ 
and $\mdiv_\Gamma  {\bf \psi} \in H^{-\frac{1}{2}}(\Gamma)$,  hence we have from \cite{ConH}
$$
\| \mcurl^2 \vectM_2({\bf \psi}) \|_{\vecL^2(D^{\pm})} \le c \left( \|{\bf \psi}\|_{\vecH_t^{-\frac{3}{2}}(\Gamma)}+ \|\mdiv_\Gamma  {\bf \psi}\|_{H^{-\frac{1}{2}}(\Gamma)}  \right).
$$
This proves the continuity property of $\vectM_1$ and $\vectM_2$. To prove the jump relations, we will use a density argument. Let 
$$
\vecu^\pm=\mcurl \int_\Gamma \Phi_k({\bf x},{\bf y}){\bf \varphi}({\bf y})  ds_y \quad \mbox{in} \quad D^{\pm}.
$$ 
We define the tangential component $\gamma_t \vecu^{\pm}$ by duality. For ${\bf \alpha} \in \vecH_t^{\frac{1}{2}}(\Gamma)$, $\|{\bf \alpha}\|_{\vecH_t^{\frac{1}{2}}(\Gamma)}=1$, there exists $\vecw^\pm \in \vecH^2(D^{\pm})$ and $\vecw^+$ compactly supported in $B_R$ such that $\gamma_t \mcurl\vecw=0, \gamma_t \vecw={\bf \alpha}$ and $\|\vecw\|_{\vecH^2(D^{\pm})} \le c\|{\bf \alpha}\|_{\vecH_t^{\frac{1}{2}}(\Gamma)}$ (see \cite{haddar}). Moreover,
$$
<{\bf \alpha,\gamma_t \vecu^\pm}>= \pm \int_{D^{\pm}} (\vecu^\pm \cdot \mcurl^2 \vecw^\pm- \vecw^\pm \cdot \mcurl^2 \vecu^\pm )d{\bf x} .
$$
Then
\begin{eqnarray*}
|<{\bf \alpha}, \gamma_t \vecu^\pm>|
&\le &  ( \|\vecu\|_{\vecL^2(D^\pm)} + \|\mcurl^2 \vecu\|_{\vecL^2(D^\pm)}) \|\vecw\|_{\vecH^2(D^{\pm})} \\
&\le & c_1 ( \|\vecu\|_{\vecL^2(D^\pm)} + \|\mcurl^2 \vecu\|_{\vecL^2(D^\pm)} ) \\
& \le & c_2\|{\bf \varphi}\|_{\vecH_t^{-\frac{1}{2}}(\Gamma)}
\end{eqnarray*}
where $c_1$ and $c_2$ are independent from $u$, therefore $\|\gamma_t \vecu^\pm\|_{\vecH_t^{-\frac{1}{2}}(\Gamma)}\le c_2\|{\bf \varphi}\|_{\vecH_t^{-\frac{1}{2}}(\Gamma)}$. Choosing ${\bf \varphi}_n \in \vecH^{-\frac{1}{2},-\frac{1}{2}}(\mbox{div}, \Gamma)$ such that ${\bf \varphi}_n \to {\bf \varphi}$ in $\vecH_t^{-\frac{1}{2}}(\Gamma)$ yields
$$
\|\gamma_t \vecu^\pm -\gamma_t \vecu_n^\pm \|_{\vecH_t^{-\frac{1}{2}}(\Gamma)} \le c\|{\bf \varphi}-{\bf \varphi}_n\|_{\vecH_t^{-\frac{1}{2}}(\Gamma)} \to 0,
$$
since $[\gamma_t \vecu_n]={\bf \varphi_n}$ for ${\bf \varphi}_n \in \vecH^{-\frac{1}{2},-\frac{1}{2}}(\mbox{div}, \Gamma)$ (see \cite{nedelec}). Letting $n \to \infty$ yields $[\gamma_t \vecu]={\bf \varphi}$ in $\vecH_t^{-\frac{1}{2}}(\Gamma)$, hence (\ref{jumprelations1}) holds. In a similar argument we can prove (\ref{jumprelations2}) (\ref{jumprelations3}).

From (\ref{jumprelations3}) we have
$$
[\gamma_{t} \mcurl \vectM_2({\bf \psi}) ] = {\bf \psi} \quad \mbox{in} \quad \vecH_t^{-\frac{3}{2}}(\Gamma).
$$
Then
$$
[\mbox{div}_{\Gamma} \gamma_t \mcurl \vectM_2({\bf \psi}) ] = \mbox{div}_{\Gamma} {\bf \psi}
$$
in the distributional sense. Notice $\mbox{div}_{\Gamma} {\bf \psi}$ and $\left(\mbox{div}_{\Gamma} \gamma_t \mcurl \vectM_2({\bf \psi})\right)^\pm$ are in $H^{-\frac{1}{2}}(\Gamma)$, then (\ref{jumprelations4}) holds.
\end{proof}\proofend

Now we are ready to prove the equivalence between the transmission eigenvalue problem and the system of integral equations (\ref{int}). Our proof follow the lines of the proof of Theorem 2.2 in \cite{ConH}.
\begin{theorem} \label{equiv-ITE}
The following statements are equivalent:
\begin{itemize}
\item[(1)] There exists non trivial $\vecE$, $\vecE_0 \in \vecL^2(D)$, $\vecE- \vecE_0 \in \vecH(\mbox{\em curl}^2, D)$ such that (\ref{maxwell1})-(\ref{neumann}) holds. 
\item[(2)] There exists non trivial $(\vecM, \vecJ) \in \vecH_t^{-\frac{1}{2}}(\Gamma) \times \vecH^{-\frac{3}{2},-\frac{1}{2}}(\mbox{\em div}, \Gamma) $ such that (\ref{int}) holds and either $\vecE_0^{\infty}(\vecM,\vecJ)=0$ or $\vecE^{\infty}(\vecM,\vecJ)=0$ where
\begin{eqnarray}
\vecE_0^{\infty}(\vecM,\vecJ)(\hat{x}) &=& \hat{x} \times\left( \frac{1}{4 \pi}\mbox{\em curl} \int_{\Gamma} \vecM(y) e^{-ik\hat{x}\cdot y} ds_y \right. \label{cr}\\
&&\hspace*{-2cm}+\left. \frac{1}{4 \pi k^2} \nabla \int_\Gamma \mbox{\em div}_{\Gamma}  \vecJ (y) e^{-ik\hat{x}\cdot y} ds_y + \int_\Gamma \vecJ( y) e^{-ik\hat{x}\cdot y}ds_y \right) \times \hat{x}\nonumber
\end{eqnarray} 
with the same expression for $\vecE^{\infty}(\vecM,\vecJ)$ where $k$ is replaced by $k_1$.
\end{itemize}
\end{theorem}
\begin{proof}
Assume $(1)$ holds. Then from the argument above (\ref{int}) we have that $\vecM$ and $\vecJ$  satisfy (\ref{int}) and hence it suffices to show $\vecE_0^{\infty}(\vecM,\vecJ)=0$ and $\vecE^{\infty}(\vecM,\vecJ)=0$. To this end, recall that $\vecE_0$ has the following representation
\begin{eqnarray}
\vecE_0(x) &=& \mcurl \int_{\Gamma} \vecM(y) \Phi_k(x,y) ds_y +\int_\Gamma \vecJ(y) \Phi_k(\cdot, y)ds_y \nonumber \\ 
&+&  \frac{1}{k^2}\nabla \int_\Gamma \mdiv_\Gamma  \vecJ( y) \Phi_k(\cdot,  y) ds_y\label{ffp}
\end{eqnarray}
where $\vecE_0 \times \nu=\vecE \times \nu =\vecM$ and $\mcurl\vecE_0 \times \nu=\mcurl\vecE \times \nu=\vecJ$. Then, from the jump relations (\ref{jumprelations1})-(\ref{jumprelations4}) of the vector potentials applied to (\ref{ffp}) and (\ref{int}) (see also \cite{ConH}), we obtain that $(\vecE_0\times \nu)^+=0$, $(\mcurl \vecE_0\times \nu)^+=0$ ($+$ denotes the traces from outside of $D$) and hence  the far field pattern $\vecE_0^{\infty}(\vecM,\vecJ)$ varnishes. The asymptotic expression of the fundamental solution $\Phi(\cdot,\cdot)$ in \cite{coltonkress} page 23,  yields (\ref{cr}). Similarly we can prove that $\vecE^{\infty}(\vecM,\vecJ)=0$. 
\\
Next assume that $(2)$ holds and  define
\begin{eqnarray}
\vecE_0(x) &=& \mcurl \int_{\Gamma} \vecM(y) \Phi_k(x,y) ds_y +\int_\Gamma \vecJ(y) \Phi_k(\cdot, y)ds_y \nonumber \\ 
&+&  \frac{1}{k^2}\nabla \int_\Gamma \mdiv_\Gamma \vecJ( y) \Phi_k(\cdot,  y) ds_y  \qquad x\in{\mathbb R}^3\setminus {\Gamma} \nonumber
\end{eqnarray}
with the same expression for $\vecE$ where $k$ is replaced by $k_1$. Again from the  jump relations of vector potentials and (\ref{int}) we have
\begin{eqnarray*}
&\mcurl \mcurl \vecE-k^2 n \vecE = 0, \quad \mcurl \mcurl \vecE_0-k^2 \vecE_0 = 0 & \qquad \qquad  \mbox{in}\quad  D\\
&\vecE \times \nu=\vecE_0 \times \nu, \quad \mcurl\vecE \times \nu=\mcurl\vecE_0 \times \nu& \qquad \qquad  \mbox{on}\quad  \Gamma
\end{eqnarray*}
(note that $\vecE$ and $\vecE_0$ are in $L^2(D)$. Therefore  it suffices to show $\vecE_0$ and $\vecE$ are non trivial. Assume to the contrary that $\vecE_0=\vecE=0$, and without loss of generality $\vecE^{\infty}(\vecM,\vecJ)=0$, then by Rellich's Lemma (see e.g. \cite{coltonkress}) $\vecE=0$ in $\R^3 \backslash \overline{D}$. Hence the jump relations imply $\vecM=0$ and $\vecJ=0$ which is a contradiction to the assumptions in $(2)$. This proves the theorem.
\end{proof} \proofend

The above discussion allows us to conclude that in order  to prove the discreteness of transmission eigenvalues we need to show that the kernel of the operator $\Lk$ is non-trivial for at most a discrete  set of wave numbers $k$.
\subsection{Properties of  the operator $\Lk$}
In the following, we will show the operator $\Lk$ is Fredholm of index zero and use the analytic Fredholm theory to obtain our main theorem. To this end we first show that for purely complex wave number $k:=i\kappa$, $\kappa>0$, $\Lk$ restricted to 
$$
 \vecH_0^{-\frac{3}{2},-\frac{1}{2}}(\mbox{div}, \Gamma):=\left\{ \vecu \in  \vecH^{-\frac{3}{2},-\frac{1}{2}}(\mbox{div}, \Gamma) ,\,\, \mdiv_{\Gamma}\vecu=0\right\}.
$$
satisfies the coercive property. In the following lemma we use the shorthand notation $\vecH_0(\Gamma):= \vecH_t^{-\frac{1}{2}}(\Gamma) \times \vecH_0^{-\frac{3}{2},-\frac{1}{2}}(\mbox{div}, \Gamma)$ and its dual space $\vecH^*(\Gamma):= \vecH_t^{\frac{1}{2}}(\Gamma) \times \left(\vecH_0^{-\frac{3}{2},-\frac{1}{2}}(\mbox{div}, \Gamma)\right)'$ where the dual  $ \left(\vecH_0^{-\frac{3}{2},-\frac{1}{2}}(\mbox{div}, \Gamma)\right)'$ of the  subspace $\vecH_0^{-\frac{3}{2},-\frac{1}{2}}(\mbox{div}, \Gamma)\subset \vecH^{-\frac{3}{2},-\frac{1}{2}}(\mbox{div}, \Gamma)$ is understood in the sense of the duality defined by Lemma \ref{dual}.
\begin{lemma} \label{coercive}
Let $\kappa>0$. The operator ${\bf L}(i\kappa):\vecH_0(\Gamma)\to \vecH^*(\Gamma)$ is strictly coercive, i.e.
\begin{eqnarray*}
\left|\left<{\bf L}(i\kappa)\left( \begin{array}{c} \vecM \\ \vecJ \end{array} \right),\left( \begin{array}{c} \vecM \\ \vecJ \end{array} \right) \right>\right|  \ge c\left(\| \vecM\|_{\vecH_t^{-\frac{1}{2}}(\Gamma)}+\|\vecJ \|_{ \vecH^{-\frac{3}{2},-\frac{1}{2}}(\mbox{\em div}, \Gamma) } \right),
\end{eqnarray*}
where $c$ is a constant depending only on $\kappa$.
\end{lemma}
\begin{proof}
We consider the following problem: for given $(\vecM, \vecJ) \in \vecH_t^{-\frac{1}{2}}(\Gamma) \times \vecH_0^{-\frac{3}{2},-\frac{1}{2}}(\mbox{div}, \Gamma) $ find $\vecU \in \vecL^2(\R^3)$, $\mcurl \vecU \in \vecL^2(\R^3)$, $\mcurl^2 \vecU \in \vecL^2(\R^3)$  such that
\begin{eqnarray}
(\mcurl^2 + n\kappa^2)(\mcurl^2 + \kappa^2) \vecU = 0 \quad &\mbox{in}& \quad \R^3 \backslash \Gamma \label{coercive_pde1} \\ \,
[ \nu \times \mcurl^2 \vecU ]= (n\kappa^2-\kappa^2)\vecM \quad &\mbox{on}& \quad \Gamma \label{coercive_pde2}\\ \,
[ \nu \times \mcurl^3 \vecU ]= (n\kappa^2-\kappa^2)\vecJ \quad &\mbox{on}& \quad \Gamma \label{coercive_pde3}
\end{eqnarray}
where $[\cdot]$ denotes the jump across $\Gamma$. Multiplying (\ref{coercive_pde1}) by a test function $\vecW$ and integrating by parts yield
\begin{eqnarray} 
&&\int_{\R^3 \backslash \Gamma} (\mcurl^2 + n\kappa^2)\vecU \cdot (\mcurl^2 + \kappa^2)\overline{\vecW}  dx\nonumber\\
&&\hspace*{3cm}= (n\kappa^2-\kappa^2)\left( \int_{\Gamma} \gamma_{\Gamma} \mcurl \overline{\vecW} \cdot \vecM ds +\int_{\Gamma} \gamma_{\Gamma}\overline{\vecW}  \cdot \vecJ ds\right) \label{coercive_varitional}
\end{eqnarray}
First we show that the right hand side is well defined. Note that $\mdiv(\mcurl \vecW)=0$, hence from \cite{nedelec}  $\mcurl \vecW \in \vecH^1(\R^3)$ and thus $\gamma_{\Gamma} \mcurl \vecW \in \vecH_t^{\frac{1}{2}}(\Gamma)$, which implies  $\int_{\Gamma} \gamma_{\Gamma} \mcurl \overline{\vecW} \cdot \vecM ds$ is defined in $\vecH_t^{\frac{1}{2}}(\Gamma)$, $\vecH_t^{-\frac{1}{2}}(\Gamma)$ duality. Since $\gamma_{\Gamma}\vecW \in \vecH_t^{-\frac{1}{2}}(\Gamma)$ and  $\mcurl_{\Gamma}\vecW=\gamma_{\Gamma} \mcurl \vecW \in \vecH_t^{\frac{1}{2}}(\Gamma)$ then from Lemma \ref{dual} $\int_{\Gamma} \gamma_{\Gamma}\overline{\vecW}  \cdot \vecJ ds$ is well defined. 

Now let 
$$
\vecV:= \{ \vecU \in \vecL^2(\R^3), \mcurl \vecU \in \vecL^2(\R^3), \mcurl^2 \vecU \in \vecL^2(\R^3)  \}
$$
equipped with the norm
$$
\| \vecU \|_{\vecV}^2= \int_{\R^3} (|\mcurl^2 \vecU|^2+|\mcurl \vecU|^2+|\vecU|^2)dx.
$$
Next taking $\vecW=\vecU$ in the continuous sesquilinear form in the left-hand side of (\ref{coercive_varitional}), and after integrating by parts  (note that $\vecU$ and $\mcurl \vecU$ are continuous across $\Gamma$, we obtain
\begin{eqnarray*}
&&\int_{\R^3 \backslash \Gamma} (\mcurl^2 + n\kappa^2)\vecU \cdot (\mcurl^2 + \kappa^2)\overline{\vecU} dx \\
=&& \int_{\R^3} (|\mcurl^2 \vecU|^2+(n\kappa^2+\kappa^2)|\mcurl \vecU|^2+n\kappa^2\kappa^2|\vecU|^2) dx\geq c\| \vecU \|_{\vecV}
\end{eqnarray*}
where $c$ is a constant depending on $\kappa$. The Lax-Milgram lemma guaranties the existence of a  unique solution to (\ref{coercive_varitional}). Up to here we did not need that $\mdiv_{\Gamma} \vecJ=0$. Next we define
\begin{eqnarray}
\vecU&=& \mcurl \int_{\Gamma} \vecM(y) (\Phi_{\sqrt{n}\kappa}(\cdot,y)-\Phi_{\kappa}(\cdot,y)) ds +\int_\Gamma \vecJ(y)  (\Phi_{\sqrt{n}\kappa}(\cdot,y)-\Phi_{\kappa}(\cdot,y)) ds \nonumber \\ 
&+& \frac{1}{(i\sqrt{n}\kappa)^2}\nabla \int_\Gamma \mdiv_\Gamma  \vecJ (y) \Phi_{\sqrt{n}\kappa}(\cdot, y) ds-\frac{1}{(i\kappa)^2}\nabla \int_\Gamma \mdiv_\Gamma  \vecJ (y) \Phi_{\kappa}(\cdot, y) ds,  \nonumber
\end{eqnarray}
then $\vecU \in \vecL^2(\R^3)$, $\mcurl \vecU \in \vecL^2(\R^3)$, $\mcurl^2 \vecU \in \vecL^2(\R^3)$ and satisfies (\ref{coercive_pde1})-(\ref{coercive_pde3}), hence $\vecU$ defined above is the unique solution to (\ref{coercive_varitional}). Now for a given $\gamma_\Gamma \mcurl {\vecW} \in \vecH^{\frac{1}{2}}(\Gamma)$, let  us construct a lifting function  $\tilde{\vecW} \in \vecH^2(\R^3)$ \cite{haddar}  such that $\gamma_\Gamma \mcurl \tilde{\vecW} =\gamma_\Gamma \mcurl {\vecW}$, $\gamma_\Gamma \tilde{\vecW}=0$ and $\| \tilde{\vecW} \|_{\vecH^2(\R^3)} \le c\|  \gamma_\Gamma \mcurl \tilde{\vecW}\|_{\vecH^{\frac{1}{2}}(\Gamma)}$ for some constant $c$. Then
\begin{eqnarray*}
&&\hspace*{-1cm}\left|\int_{\Gamma} \gamma_{\Gamma} \mcurl \vecW \cdot \vecM ds\right| =\left|\int_{\Gamma} \gamma_{\Gamma} \mcurl \tilde\vecW \cdot \vecM ds\right|\\
&&=\frac{1}{|n\kappa^2-\kappa^2|}\left|\int_{\R^3 \backslash \Gamma} (\mcurl^2 + n\kappa^2)\vecU \cdot (\mcurl^2 + \kappa^2)\tilde \vecW dx\right| \\
&&\leq\|\vecU\|_{\vecV} \|\tilde{\vecW}\|_{\vecV}\le c\|\vecU\|_{\vecV} \|  \gamma_\Gamma \mcurl \tilde{\vecW}\|_{\vecH^{\frac{1}{2}}(\Gamma)}.
\end{eqnarray*}
Hence $\|\vecM\|_{\vecH_t^{-\frac{1}{2}}(\Gamma)}\le c \|\vecU\|_{\vecV}$. Similarly for given $\gamma_\Gamma \vecW \in \vecH^{\frac{3}{2}}(\Gamma)$ we construct the lifting $\tilde{\vecW}_2 \in \vecH^2(\R^3)$ \cite{haddar}  such that $\gamma_\Gamma \tilde{\vecW}_2 =\gamma_\Gamma \vecW$, $\gamma_\Gamma \mcurl \tilde{\vecW_2}=0$ and $\| \tilde{\vecW_2} \|_{\vecH^2(\R^3)} \le c\|  \gamma_T \tilde{\vecW_2}\|_{\vecH^{\frac{3}{2}}(\Gamma)}$ for some  constant $c$. We recall that $\mdiv_{\Gamma} \vecJ=0$ hence  from the Helmoltz decomposition $\vecJ=\overrightarrow{\mcurl}_{\Gamma} q \in \vecH^{-\frac{3}{2}}(\Gamma)$. Thus we have
\begin{eqnarray*}
&&\hspace*{-1cm}\left|\int_{\Gamma} \gamma_{\Gamma} \vecW \cdot \vecJ ds\right| =\left|\int_{\Gamma} \gamma_{\Gamma} \tilde{\vecW}_2 \cdot \vecJ ds\right| \\
&&=\frac{1}{|n\kappa^2-\kappa^2|}\left|\int_{\R^3 \backslash \Gamma} (\mcurl^2 + n\kappa^2)\vecU \cdot (\mcurl^2 + \kappa^2)\tilde{\vecW}_2 dx\right| \\
&&\le c\|\vecU\|_{\vecV} \|\tilde{\vecW}_2\|_{\vecV} \le c\|\vecU\|_{\vecV} \|  \gamma_T \vecW\|_{\vecH^{\frac{3}{2}}(\Gamma)}.
\end{eqnarray*}
Since $\vecJ=\overrightarrow{\mcurl}_{\Gamma} q \in \vecH^{-\frac{3}{2}}(\Gamma)$, then by duality $\|\vecJ\|_{ \vecH_0^{-\frac{3}{2},-\frac{1}{2}}(\mbox{div}, \Gamma)}\le c \|\vecU\|_{\vecV}$.

Finally 
\begin{eqnarray*}
&&\hspace*{-1cm}\left|\left<{\bf L}(i\kappa)\left( \begin{array}{c} \vecM \\ \vecJ \end{array} \right),\left( \begin{array}{c} \vecM \\ \vecJ \end{array} \right)\right>\right| \\
&&= \left \| \int_{\Gamma} \gamma_{\Gamma} \mcurl \vecU \cdot \overline{\vecM} ds +\int_{\Gamma} \gamma_{\Gamma}\vecU \cdot \overline{\vecJ} ds \right \| \\
&&\geq c\|\vecU\|_{\vecV} \geq c \left(\| \vecM\|_{\vecH_t^{-\frac{1}{2}}(\Gamma)}+\|\vecJ \|_{ \vecH^{-\frac{3}{2},-\frac{1}{2}}(\mbox{div}, \Gamma) } \right).
\end{eqnarray*}
where $c$ is a constant depending on $\kappa$. This proves our lemma.
\end{proof} \proofend

Next we proceed with the following lemma.
\begin{lemma} \label{compact}
Let $\gamma(k):=\frac{k_1^2-k^2}{|k_1|^2-|k|^2}$ and $k_1=k\sqrt{n}$ for $k\in{\mathbb C}\setminus {\mathbb R}_{-}$. Then $\Lk + \gamma(k) \vecL (i|k|): \vecH_0(\Gamma)\to \vecH^*(\Gamma)$ is compact.
\end{lemma}
\begin{proof}
From \cite{ConH} Theorem 3.8,   the operator
$$
(\Skb-\Ska)+\gamma(k)({\bf S}_{i|k_1|}-{\bf S}_{i|k|}): \vecH^{-\frac{3}{2}}(\Gamma) \to \vecH^{\frac{3}{2}}(\Gamma)
$$
is compact. Then from (\ref{prop}) we have
$$
\nabla_{\Gamma} \circ (S_{k_1}-S_{k}) \circ \mdiv_{\Gamma} + \gamma(k) \nabla_{\Gamma} \circ ({S}_{i|k_1|}-{S}_{i|k|}) \circ \mdiv_{\Gamma} : \vecH^{-\frac{1}{2}}(\Gamma) \to \vecH^{\frac{1}{2}}(\Gamma) 
$$
$$
(\Kkb-\Kka)+\gamma(k) ({\bf K}_{i|k_1|}-{\bf K}_{i|k|}) : \vecH^{-\frac{3}{2}}(\Gamma) \to \vecH^{\frac{1}{2}}(\Gamma) 
$$
$$
(\Kkb-\Kka)+\gamma(k) ({\bf K}_{i|k_1|}-{\bf K}_{i|k|})  : \vecH^{-\frac{1}{2}}(\Gamma) \to \vecH^{\frac{3}{2}}(\Gamma) 
$$
$$
\left(\frac{1}{k_1}\Kkb-\frac{1}{k}\Kka\right)+\gamma(k)\left(\frac{1}{i|k_1|}{\bf K}_{i|k_1|}-\frac{1}{i|k|}{\bf K}_{i|k|}\right)  : \vecH^{-\frac{3}{2}}(\Gamma) \to \vecH^{\frac{3}{2}}(\Gamma) 
$$
are compact. It remains to show that  
$$
({k_1}^2\Skb-k^2\Ska)+\gamma(k) ((i|k_1|)^2{\bf S}_{i|k_1|}-(i|k|)^2{\bf S}_{i|k|})  : \vecH^{-\frac{1}{2}}(\Gamma) \to \vecH^{\frac{1}{2}}(\Gamma)$$
is compact. Since 
\begin{eqnarray*}
&&({k_1}^2\Skb-k^2\Ska)+\gamma(k) ((i|k_1|)^2{\bf S}_{i|k_1|}-(i|k|)^2{\bf S}_{i|k|}) \\
=&&({k_1}^2(\Skb-{\bf S}_0)-k^2(\Ska-{\bf S}_0))+\gamma(k) ((i|k_1|)^2({\bf S}_{i|k_1|}-{\bf S}_0)-(i|k|)^2({\bf S}_{i|k|}-{\bf S}_0))
\end{eqnarray*}
and $\Ska-{\bf S}_0$ is compact, then the compactness follows. Hence the proof of the lemma is completed.
\end{proof} \proofend

In order to handle the non divergence free part of  $\vecJ$, we will split  $\vecJ:= \vecQ+\vecP$ where  $\vecQ \in  \vecH_0^{-\frac{3}{2},-\frac{1}{2}}(\mbox{div}, \Gamma)$, $\vecP=\nabla_{\Gamma} p \in \vecH_t^{\frac{1}{2}}(\Gamma)$ and rewrite the equation (\ref{int}) for the unknowns $(\vecM, \vecQ, \vecP)$.  To this end let us define
$$\vecH_1(\Gamma):=\left\{ \vecP \in \vecH_t^{\frac{1}{2}}(\Gamma), \mcurl_\Gamma \vecP=0 \right\}$$
and  introduce the operator 
\begin{align} \label{int_tLk}
\tLk=
\left( \begin{array}{ccc}
k_1\Tkb-k\Tka & \Kkb-\Kka & \Kkb-\Kka \\
\Kkb-\Kka & \Skb-\Ska & \Skb-\Ska \\
\Kkb-\Kka & \Skb-\Ska & (\Skb-\Ska)+\nabla_{\Gamma} \circ (\frac{1}{k_1^2}\Skb-\frac{1}{k^2}\Ska) \circ \mdiv_{\Gamma}
\end{array} \right).
\end{align}
From from Lemma \ref{dual} and Lemma \ref{bddLk}  $\tLk:\vecH_0(\Gamma)\times \vecH_1(\Gamma)\to \vecH^*(\Gamma)\times \vecH^{-\frac{1}{2}}(\Gamma)$ is  bounded  and furthermore the family of operators $\tLk$ depends analytically on $k \in \C \backslash \R_{-}$, where recall $\vecH_0(\Gamma):= \vecH_t^{-\frac{1}{2}}(\Gamma) \times \vecH_0^{-\frac{3}{2},-\frac{1}{2}}(\mbox{div}, \Gamma)$ with its dual $\vecH^*(\Gamma)$.
We first notice that (\ref{int}) is equivalent to the following:
\begin{eqnarray*}
\left< \Lk \left( \begin{array}{c} \vecM \\ \vecJ \end{array} \right),\left( \begin{array}{c} \tilde\vecM \\ \tilde \vecJ \end{array} \right) \right> =0,
\end{eqnarray*}
for any $(\tilde \vecM, \tilde \vecJ) \in \vecH_t^{\frac{1}{2}}(\Gamma) \times \vecH^{-\frac{1}{2},\frac{1}{2}}(\mbox{curl}, \Gamma) $ which equivalently can be written as 
\begin{eqnarray*}
\left< \tLk \left( \begin{array}{c} \vecM \\ \vecQ \\ \vecP \end{array} \right),\left( \begin{array}{c} \tilde\vecM\\ \tilde\vecQ \\ \tilde\vecP \end{array} \right) \right> =0,
\end{eqnarray*}
for any $(\tilde \vecM, \tilde \vecQ, \tilde \vecP) \in \vecH^*\times \vecH_t^{-\frac{1}{2}}(\Gamma)$.
Now we are ready to prove the following lemma.
\begin{lemma} \label{fredholm}
The operator   $\tLk:\vecH_0(\Gamma)\times \vecH_1(\Gamma)\to \vecH^*(\Gamma)\times \vecH^{-\frac{1}{2}}(\Gamma)$ is Fredholm with index zero, i.e. it can be written as a sum of an invertible operator and a compact operator.
\end{lemma}
\begin{proof}
We rewrite the operator $\tLk$ as follows 
\begin{align} \label{int_tLk}
\tLk&=
-\left( \begin{array}{ccc}
\gamma(k)(i|k_1|{\bf T}_{i|k_1|}-i|k|{\bf T}_{i|k|} )&\gamma(k)( {\bf K}_{i|k_1|}-{\bf K}_{i|k|} )& 0 \\
\gamma(k)( {\bf K}_{i|k_1|}-{\bf K}_{i|k|} )& \gamma(k) ({\bf S}_{i|k_1|}-{\bf S}_{i|k|} )& 0 \\
0 & 0 & \nabla_{\Gamma} \circ (-\frac{1}{k_1^2}+\frac{1}{k^2}) {\bf S}_0 \circ \mdiv_{\Gamma}
\end{array} \right)  \nonumber \\
&+
\left( \begin{array}{ccc}
\gamma(k)\left(i|k_1|{\bf T}_{i|k_1|}-i|k|{\bf T}_{i|k|}\right) & \gamma(k)\left( {\bf K}_{i|k_1|}-{\bf K}_{i|k|} \right) & 0 \\
\gamma(k)\left( {\bf K}_{i|k_1|}-{\bf K}_{i|k|} \right) & \gamma(k)\left( {\bf S}_{i|k_1|}-{\bf S}_{i|k|} \right) & 0 \\
0 & 0 &  \nabla_{\Gamma} \circ (-\frac{1}{k_1^2}+\frac{1}{k^2}) {\bf S}_0 \circ \mdiv_{\Gamma}
\end{array} \right)  \nonumber \\
&+\left( \begin{array}{ccc}
k_1\Tkb-k\Tka & \Kkb-\Kka & \Kkb-\Kka \\
\Kkb-\Kka & \Skb-\Ska & \Skb-\Ska \\
\Kkb-\Kka & \Skb-\Ska & (\Skb-\Ska)+\nabla_{\Gamma} \circ (\frac{1}{k_1^2}\Skb-\frac{1}{k^2}\Ska) \circ \mdiv_{\Gamma}
\end{array} \right)   \nonumber \\
&=:
\tilde{\bf L}_1(k)+\tilde{\bf L}_2(k)
\end{align}
where $\tilde{\bf L}_1(k)$ is the first operator and $\tilde{\bf L}_2(k)$ is the sum of the last two operators. Then from Lemma \ref{compact} and the fact that $\Skb-\Ska$,$\Kkb-\Kka$ are smoothing operators of order 3,2 respectively, we have $\tilde{\bf L}_2(k)$ is  compact. From Lemma \ref{coercive} and  the fact that ${\bf S}_0$ is invertible, whence we have $\tilde{\bf L}_1(k)$ is invertible. This proves our lemma.
\end{proof} \proofend

\section{The case  when $N-I$ changes sign}\label{piece}
In this section we will discuss the Fredholm properties of $\Lk$ when $N$ is not a constant any longer. Our approach to handle the more general case follows exactly the lines of the discussion in Section 4 of \cite{ConH}, and here for sake of the reader's convenience we sketch the main steps of the analysis. 
\subsection{Piecewise homogeneous medium}
To begin with, we assume that $D=\overline{D}_1\cup\overline{D}_2$ such that $D_1\subset D$ and $D_2:=D\setminus\overline{D}_1$ and consider the simple case when $N=n_2I$ in $D_2$ and $N=n_1I$ in $D_1$ where $n_1>0$,  $n_2>0$ are two positive constants such that $(n_1-1)(n_2-1)<0$. Let $\Gamma=\partial D$, $\Sigma=\partial D_1$ which are assumed to be $C^2$ smooth surfaces and $\nu$ denotes the unit normal vector to either $\Gamma$ or $\Sigma$ outward to $D$ and $D_1$ respectively (see Figure \ref{areatev}). Let us recall the notations $k_1=k\sqrt{n_1}$ and $k_2=k\sqrt{n_2}$. 
\begin{figure}[ht!]
\centering
\includegraphics[scale=0.7]{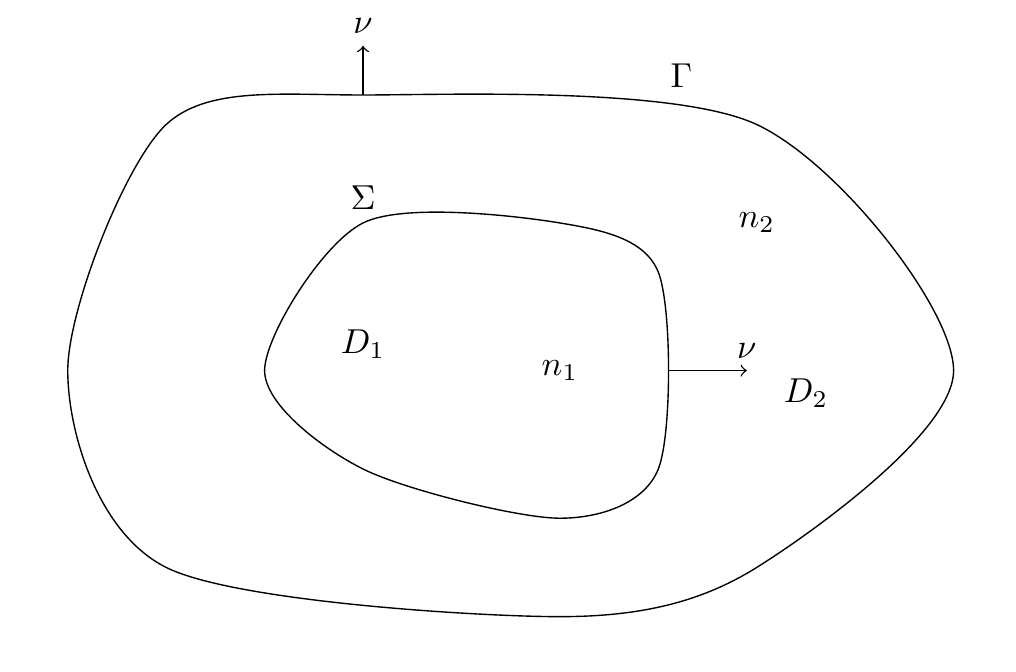}\\
\caption{Configuration of the geometry for two constants}
\label{areatev}
\end{figure}

For convenience, we let  ${\bf K}^{\Sigma, \Gamma}_k$ and ${\bf T}^{\Sigma, \Gamma}_k$ be the potentials ${\bf K}_k$ and ${\bf T}_k$ given by (\ref{dT}) and (\ref{dK}) for densities defined  on $\Sigma$ and evaluated on $\Gamma$. 
The solution of the transmission eigenvalue problem (\ref{maxwell1g})-(\ref{neumanng}) by means of the Stratton-Chu formula can be represented as 
\begin{eqnarray}
\vecE_0(x) &=& \mcurl \int_{\Gamma} (\vecE_0 \times \nu)({\bf y}) \Phi_k(x,y) ds_y +\int_\Gamma (\mcurl\vecE_0 \times \nu)({\bf y}) \Phi_k(\cdot,{\bf y})ds_y \nonumber \\ 
&+&  \frac{1}{k^2}\nabla \int_\Gamma \mdiv_T  (\mcurl\vecE_0 \times \nu)({\bf y}) \Phi_k(\cdot, {\bf y}) ds_y \quad \mbox{in} \quad D\label{an1}
\end{eqnarray}

\begin{eqnarray}
\vecE(x) &=& \mcurl \int_{\Sigma} (\vecE \times \nu)({\bf y}) \Phi_{k_1}(x,y) ds_y +\int_\Sigma (\mcurl\vecE \times \nu)({\bf y}) \Phi_{k_1}(\cdot,{\bf y})ds_y \nonumber \\ 
&+&  \frac{1}{k_1^2}\nabla \int_\Sigma \mdiv_T  (\mcurl\vecE \times \nu)({\bf y}) \Phi_{k_1}(\cdot, {\bf y}) ds_y \quad \mbox{in} \quad D_1\label{an2}
\end{eqnarray}
\begin{eqnarray}
\vecE(x) &=& \mcurl \int_{\Gamma} (\vecE \times \nu)({\bf y}) \Phi_{k_2}(x,y) ds_y +\int_\Gamma (\mcurl\vecE \times \nu)({\bf y}) \Phi_{k_2}(\cdot,{\bf y})ds_y \nonumber \\ 
&+&  \frac{1}{k_2^2}\nabla \int_\Gamma \mdiv_T  (\mcurl\vecE \times \nu)({\bf y}) \Phi_{k_2}(\cdot, {\bf y}) ds_y \nonumber \\
&-&\mcurl \int_{\Sigma} (\vecE \times \nu)({\bf y}) \Phi_{k_2}(x,y) ds_y -\int_\Sigma (\mcurl\vecE \times \nu)({\bf y}) \Phi_{k_2}(\cdot,{\bf y})ds_y \nonumber \\ 
&-&  \frac{1}{k_2^2}\nabla \int_\Sigma \mdiv_T  (\mcurl\vecE \times \nu)({\bf y}) \Phi_{k_2}(\cdot, {\bf y}) ds_y \quad \mbox{in} \quad D_2 \label{an3}
\end{eqnarray}
Let $\vecE \times \nu=\vecE_0 \times \nu =\vecM$, $\mcurl\vecE \times \nu=\mcurl\vecE_0 \times \nu=\vecJ$ on $\Gamma$ and $\vecE \times \nu=\vecM'$, $\mcurl\vecE \times \nu=\vecJ'$ on $\Sigma$. From the jump relations of the boundary integral operators across  $\Gamma$ and $\Sigma$, we have that 
\begin{align} 
\left( \begin{array}{cc}
k_2\Tkc^{\Gamma}-k\Tka^{\Gamma} & \Kkc^{\Gamma}-\Kka^{\Gamma}\\
\Kkc^{\Gamma}-\Kka^{\Gamma} & \frac{1}{k_2}\Tkc^{\Gamma}-\frac{1}{k}\Tka^{\Gamma}\end{array} \right) 
\left( \begin{array}{c} \vecM \\ \vecJ \end{array} \right)
&=
\left( \begin{array}{cc}
k_2\Tkc^{\Sigma, \Gamma} & \Kkc^{\Sigma, \Gamma}\\
\Kkc^{\Sigma, \Gamma} & \frac{1}{k_2}\Tkc^{\Sigma, \Gamma}\end{array} \right) 
\left( \begin{array}{c} \vecM' \\ \vecJ' \end{array} \right)\label{eq1}
\\
\left( \begin{array}{cc}
k_2\Tkc^{\Sigma}+k_1\Tkb^{\Sigma} & \Kkc^{\Sigma}+\Kkb^{\Sigma}\\
\Kkc^{\Sigma}+\Kkb^{\Sigma} & \frac{1}{k_2}\Tkc^{\Sigma}+\frac{1}{k_1}\Tkb^{\Sigma}\end{array} \right) 
\left( \begin{array}{c} \vecM' \\ \vecJ' \end{array} \right)
&=
\left( \begin{array}{cc}
k_2\Tkc^{\Gamma,\Sigma} & \Kkc^{\Gamma,\Sigma}\\
\Kkc^{\Gamma,\Sigma} & \frac{1}{k_2}\Tkc^{\Gamma,\Sigma}\end{array} \right) 
\left( \begin{array}{c} \vecM \\ \vecJ \end{array} \right).\label{eq2}
\end{align}
Let us denote by $\vecL_{20}(k)$, $\vecL^{\Sigma, \Gamma}(k)$, $\vecL_{21}(k)$, $\vecL^{\Gamma, \Sigma}(k)$ the matrix-valued operators in the above two equations in the order from the left to the right from the top to the bottom, respectively. By the regularity of the solution of the Maxwell's equations inside $D_2$ (see e.g. \cite{kirsch-h}), we have $(\vecM', \vecJ') \in \vecH_t^{-\frac{1}{2}}(\Sigma,\mdiv) \times \vecH_t^{-\frac{1}{2}}(\Sigma,\mdiv)$. Then the equation 
\begin{align*} 
\vecL_{21}(k)
\left( \begin{array}{c} \vecM' \\ \vecJ' \end{array} \right)= \left( \begin{array}{c} \bf g \\ \bf h \end{array} \right)
\end{align*}
where $(\bf g, \bf h) \in \vecH_t^{-\frac{1}{2}}(\Sigma,\mdiv) \times \vecH_t^{-\frac{1}{2}}(\Sigma,\mdiv)$
corresponds to the transmission problem which is to find $(\vecE_2, \vecE_1) \in \vecH_{loc}(\mcurl, \R^3 \backslash \overline{D_1}) \times \vecH(\mcurl, D_1)$ and $\vecE_2$  such that
\begin{eqnarray*}
\mcurl \mcurl \vecE_2-k_2^2 \vecE_2 = 0 \quad &\mbox{in}& \quad \R^3 \backslash \overline{D_1}   \\
\mcurl \mcurl \vecE_1-k_1^2 \vecE_1 = 0 \quad &\mbox{in}& \quad D_1  \\
\nu \times \vecE_2 - \nu \times \vecE_1=\bf g \quad &\mbox{on}& \quad \Sigma  \\ 
\nu \times (\mcurl \vecE_2) - \nu \times (\mcurl \vecE_1)=\bf h \quad &\mbox{on}& \quad \Sigma 
\end{eqnarray*}
and $\vecE_2$ satisfies the Silver-Mueller radiation condition. By well-posedeness of the transmission problem we have $\vecL_{21}(k)$ is invertible. 
Hence pugging in (\ref{eq1}) ${\bf M}'$ and ${\bf J}'$ from (\ref{eq2}) we obtain the following equation for ${\bf M}$ and ${\bf J}$ 
\begin{align} \label{int-n} 
\vecL(k)
\left( \begin{array}{c} \vecM \\ \vecJ \end{array} \right) = \left( \begin{array}{c} 0 \\ 0 \end{array} \right)
\end{align}
where $\vecL(k):=\vecL_{20}(k)-\vecL^{\Sigma, \Gamma}(k) {\vecL_{21}(k)}^{-1} \vecL^{\Gamma, \Sigma}(k) $. Then in a similar way to Theorem \ref{equiv-ITE}, we can prove the following theorem.
\begin{theorem} \label{equiv-ITE-n}
The following statements are equivalent:
\begin{itemize}
\item[(1)] There exists non trivial $\vecE, \vecE_0 \in L^2(D), \vecE- \vecE_0 \in \vecH(\mbox{\em curl}^2, D)$ such that (\ref{maxwell1})-(\ref{neumann}) holds. \\
\item[(2)] There exists non trivial $(\vecM, \vecJ) \in \vecH_t^{-\frac{1}{2}}(\Gamma) \times \vecH^{-\frac{3}{2},-\frac{1}{2}}(\mbox{\em div}, \Gamma) $ such that (\ref{int-n}) holds and $\vecE_0^{\infty}(\vecM,\vecJ)=0$  where
\begin{eqnarray*}
\vecE_0^{\infty}(\vecM,\vecJ)(\hat{x}) &=& \hat{x} \times\left( \frac{1}{4 \pi}\mbox{\em curl} \int_{\Gamma} \vecM(y) e^{-ik\hat{x}\cdot y} ds_y \right.\\
&+&\left. \frac{1}{4 \pi k^2} \nabla \int_\Gamma \mbox{\em div}_\Gamma  \vecJ (y) e^{-ik\hat{x}\cdot y} ds_y + \int_\Gamma \vecJ( y) e^{-ik\hat{x}\cdot y}ds_y\right) \times \hat{x}
\end{eqnarray*} 
\end{itemize}
\end{theorem}

Now we note that $\Sigma$ and $\Gamma$ are two disjoint curves and  hence we have that $\vecL^{\Sigma, \Gamma}(k)$, $\vecL^{\Gamma, \Sigma}(k)$ are compact. By writing $\vecL(k)$ as a $3 \times 3$ matrix operator $\tLk$ similar to (\ref{int_tLk}), we can have the following lemma directly from Lemma \ref{fredholm}. 
\begin{lemma} \label{fredholm-n}
The operator $\tLk:\vecH_0(\Gamma)\times \vecH_1(\Gamma)\to \vecH^*(\Gamma)\times \vecH^{-\frac{1}{2}}(\Gamma)$ is  Fredholm with index zero, i.e. it can be written as a sum of an invertible operator and a compact operator. Furthermore the family of the operators $\tLk$  depends analytically on $k \in \C \backslash \R_{-}$.
\end{lemma}
This approach can be readily generalized to the case when the medium consists of finitely many homogeneous layers.
\subsection{General  inhomogeneous medium} \label{gen}

In a more general case where $N=n(x)I$ in $D_1$, where $n\in L^{\infty}(D_1)$ such that $n(x)\geq \alpha>0$ but still constant in $D_2$, we can  prove the same result as in  Lemma \ref{fredholm-n} by replacing fundamental solution $\Phi_{k_1}(\cdot,y)$ with the free space fundamental ${\mathbb G}(\cdot, y)$ of 
$$\Delta {\mathbb G}(\cdot,y)+k^2n(x){\mathbb G}(\cdot,y)=-\delta_y\qquad \mbox{in}\; {\mathbb R}^3$$
in the distributional sense together with the Sommerfeld radiation condition, where $n(x)$ is extended by its constant value in $D_2$ to the whole space ${\mathbb R}^3$. Because $\Phi_{k_2}(\cdot,y)-{\mathbb G}(\cdot, y)$ solves the Helmholtz equation with wave number $k_2$ in the neighborhood of $\Gamma$ the mapping properties of the integral operators do not change. We refer the reader to Section 4.2 of \cite{ConH} for more details. 

In fact the above idea can be applied even in a more general case, provided that $N$ is a  positive constant not equal to one  in a neighborhood of  $\Gamma$. More precisely, consider  a neighborhood ${\mathcal O}$ of $\Gamma$ in $D$ (above denoted by $D_2$) with $C^2$ smooth boundary (e.g.  one can take ${\mathcal O}$ to be the region in $D$ bounded by $\Gamma$ and $\Sigma:=\left\{x-\epsilon \nu(x), \; x\in \Gamma\right\}$ for some $\epsilon>0$ where $\nu$ is the outward unit normal vector to $\Gamma$). Assume that $N=nI$ in ${\mathcal O}$, where $n\neq 1$ is a positive constant, whereas   in $D\setminus \overline{\mathcal O}$  $N$ satisfies the assumptions at the beginning of the paper, i.e. $N$ is a $3\times3$  matrix-valued function  with $L^\infty(D)$ entries such that $\overline{\xi} \cdot \re(N)\xi\geq \alpha>0$ and $\overline{\xi} \cdot \im(N)\xi\geq 0$ for every $\xi\in {\mathbb C}^3$.   Then similar result as in Theorem  \ref{equiv-ITE-n} and Lemma \ref{fredholm-n} holds true in this case. Indeed, without going into details, we can express $\vecE_0$ by (\ref{an1}) and $\vecE$ by (\ref{an3}) in ${\mathcal O}$ and in $D\setminus \overline{\mathcal O}$ we can leave it in the form of a partial differential equation with Cauchy data connected to $\vecE$ in ${\mathcal O}$. Hence it  is possible to obtain an equation of the form (\ref{int-n}) where the operator $\vecL(k)$ is written as
\begin{equation} \label{integralA}
\vecL(k)=\vecL_{n}(k)-\vecL^{\Sigma, \Gamma}(k) {{\mathbf A}^{-1}(k)} \vecL^{\Gamma, \Sigma}(k)
\end{equation}
where $\vecL_{n}(k)$ is the boundary integral operator corresponding to the transmission eigenvalue problem with contrast $n-1$, the compact operators $\vecL^{\Sigma, \Gamma}(k)$ and $\vecL^{\Gamma, \Sigma}(k)$ are  defined   right below (\ref{eq1}) and (\ref{eq2})  and ${\mathbf A}(k)$ is the invertible solution operator corresponding to the well-posed transmission problem
\begin{eqnarray*}
\mcurl \mcurl \vecE_2-k^2n_2 \vecE_2 = 0 \quad &\mbox{in}& \quad \R^3 \backslash \overline{D_1}   \\
\mcurl \mcurl \vecE_1-k^2 N \vecE_1 = 0 \quad &\mbox{in}& \quad D_1  \\
\nu \times \vecE_2 - \nu \times \vecE_1=\bf g \quad &\mbox{on}& \quad \Sigma  \\ 
\nu \times (\mcurl \vecE_2) - \nu \times (\mcurl \vecE_1)=\bf h \quad &\mbox{on}& \quad \Sigma 
\end{eqnarray*}
and $\vecE_2$ satisfies the Silver-M{\"u}ller radiation condition.  Hence the above analysis can apply to prove   analogues Theorem  \ref{equiv-ITE-n} and Lemma \ref{fredholm-n}.

For later use in the following we formally state the assumptions on $N$ (here ${\mathcal O}$ is a neighborhood of $\Gamma$ as explained above).
\begin{assum}\label{massum}
$N$ is a $3\times 3$ symmetric matrix-valued function with $L^\infty(D)$  entries such that $\overline{\xi} \cdot \re(N)\xi\geq \alpha>0$ and $\overline{\xi} \cdot \im(N)\xi\geq 0$ for every $\xi\in {\mathbb C}^3$, $|\xi|=1$ and $N=nI$ in $\mathcal{O}$ where $n\neq 1$ is a positive constant.  
\end{assum}
\section{The existence of  non transmission eigenvalue wave numbers}
In this section we assume that $N$ satisfies Assumption \ref{massum}  and  consider pure imaginary wave numbers $k$ and, for convenience, let $\lambda:=-k^2$ be a real positive number and start by proving an a priori estimate following the idea of \cite{syld} for the scalar case.
\begin{lemma}\label{lema1}
Assume that $N$ satisfies \ref{massum} and $\chi(x) \in \C^\infty_0(D)$ is real valued cutoff function with $0 \le \chi \le 1$ and $\chi\equiv 1$ in $D \backslash \overline{\mathcal{O}}$. If $\vecv \in \vecL^2(D)$ and 
$$
(\mbox{\em curl}\,\mbox{\em curl}  + \lambda) \vecv =0 \quad \mbox{in} \quad D
$$ 
then there exists a constant $K(\chi)$ such that for sufficiently large $\lambda$
\begin{equation} \label{chiv}
\| \chi \vecv \|^2 \le K \frac{\| (1-\chi)\vecv \|^2}{\lambda}.
\end{equation}
\end{lemma}
\begin{proof}
Since $\chi \in \C^\infty_0(D)$ we have
\begin{eqnarray*}
0&=&\int_D (\mcurl \mcurl + \lambda) \vecv \cdot (\chi^2 \overline{\vecv}) dx= \int_D \mcurl \mcurl \vecv  \cdot (\chi^2 \overline{\vecv}) dx+ \lambda \int_D \vecv \cdot (\chi^2 \overline{\vecv}) dx\\
&=& \int_D \mcurl \vecv  \cdot \mcurl (\chi^2 \overline{\vecv})dx + \lambda \int_D \vecv \cdot (\chi^2 \overline{\vecv})dx \\
&=& \int_D \mcurl \vecv  \cdot (\chi \mcurl(\chi \overline{\vecv}))dx + \int_D \mcurl \vecv  \cdot (\nabla \chi  \times (\chi \overline{\vecv}))dx + \lambda \int_D  \vecv \cdot (\chi^2 \overline{\vecv})dx \\
&=& \int_D \mcurl (\chi \vecv)  \cdot \mcurl(\chi \overline{\vecv}) dx - \int_D \mcurl (\chi \overline{\vecv})  \cdot (\nabla \chi  \times \vecv)dx \\
&+& \int_D \mcurl \vecv  \cdot (\nabla \chi  \times (\chi \overline{\vecv}))dx +\lambda \int_D \vecv \cdot (\chi^2 \overline{\vecv})dx \\
&=& \int_D |\mcurl (\chi \vecv)|^2dx - \int_D ( \chi \mcurl \overline{\vecv}+\nabla \chi \times  \overline{\vecv})  \cdot (\nabla \chi  \times \vecv)dx \\
&+& \int_D \mcurl \vecv  \cdot (\nabla \chi  \times (\chi \overline{\vecv}))dx +\lambda \int_D  \vecv \cdot (\chi^2 \overline{\vecv})dx \\
&=& \int_D |\mcurl (\chi \vecv)|^2dx -\int_D |(\nabla \chi  \times \vecv)|^2dx+\lambda \int_D |\chi\vecv|^2dx \\
&+& \int_D \left( (\chi \mcurl \vecv) \cdot (\nabla \chi  \times \overline{\vecv}) -  (\chi \mcurl \overline{\vecv}) \cdot (\nabla \chi  \times \vecv)\right)dx.
\end{eqnarray*}
Taking the real part yields
$$
\int_D |\mcurl (\chi \vecv)|^2dx + \lambda \int_D |\chi\vecv|^2dx =\int_D |(\nabla \chi  \times \vecv)|^2dx
$$
and then
\begin{eqnarray*}
\lambda \| \chi\vecv \|^2 \le K(\chi) \| \vecv \|^2 \le K(\chi) \left( \| \chi\vecv \|^2 + \| (1-\chi)\vecv \|^2\right) 
\end{eqnarray*}
which yields (\ref{chiv}) for sufficiently large $\lambda$.

\end{proof} \proofend

Now we are ready to prove the following theorem.
\begin{theorem} \label{nonITE}
Under the assumption \ref{massum}, there exists a sufficiently large real $\lambda>0$ where $\lambda=-k^2$ such that (\ref{maxwell1g})-(\ref{neumanng}) has only trivial solutions.
\end{theorem}
\begin{proof}
Assume first  $n-1<0$ in $\mathcal{O}$, let $\vecu=\vecE-\vecE_0 \in \vecH_0(\mcurl^2, D)$, $\vecv=\lambda \vecE_0 \in \vecL^2(D)$, then
\begin{eqnarray}
\mcurl \mcurl \vecu+\lambda N \vecu = -(N-I)\vecv \quad &\mbox{in}& \quad D  \label{1maxwell1} \\
\mcurl \mcurl \vecv+\lambda \vecv = 0 \quad &\mbox{in}& \quad D \label{1maxwell} \\
\nu \times \vecu = \nu \times (\mcurl \vecu) =0 \quad &\mbox{on}& \quad \Gamma. \label{1dirichlet}
\end{eqnarray}
Then for any ${\bf \varphi} \in {\bf C}_0^\infty(D)$, interpreting (\ref{1maxwell}) in the distributional sense yields
$$
\int_D \vecv(\mcurl \mcurl {\bf \varphi}+\lambda {\bf \varphi})=0,
$$
and hence the denseness of ${\bf C}_0^\infty(D)$ in $\vecH_0(\mcurl^2, D)$ (see \cite{haddar}) yields
\begin{equation} \label{case1-1}
\int_D \overline{\vecv} \cdot \mcurl^2 \vecu + \lambda \int_D \overline{\vecv} \cdot \vecu=0
\end{equation}
Multiplying (\ref{1maxwell1}) by $\overline{\vecv}$ yields
$$
\int_D \overline{\vecv} \cdot \mcurl^2 \vecu dx + \lambda \int_D N \vecu \cdot \overline{\vecv} dx + \int_D (N-I)\vecv \cdot \overline{\vecv}dx=0
$$
Combining the above with (\ref{case1-1}) yields
\begin{equation} \label{case1-2}
\lambda \int_D (N-I)\overline{\vecu} \cdot \vecv dx + \int_D (N-I)\vecv \cdot \overline{\vecv} dx=0
\end{equation}
Multiplying (\ref{1maxwell1}) by $\overline{\vecu}$ and integrating by parts yields
\begin{equation*} 
\int_D |\mcurl \vecu|^2 dx + \lambda \int_D N \vecu \cdot \overline{\vecu} dx + \int_D (N-I)\vecv \cdot \overline{\vecu} dx=0
\end{equation*}
Noting that $N$ is symmetric, we have $(N-I)\overline{\vecu} \cdot \vecv=(N-I)\vecv \cdot \overline{\vecu}$ and hence
\begin{equation} \label{case1-3}
\int_D |\mcurl \vecu|^2 dx + \lambda \int_D N \vecu \cdot \overline{\vecu} dx + \int_D (N-I)\overline{\vecu} \cdot \vecv dx=0
\end{equation}
By regularity \cite{nedelec} $\vecv$ is sufficiently smooth in $D$ away from the boundary and hence  by unique continuation we can see $\int_\mathcal{O} (n-1)(1-\chi^2)|\vecv|^2 dx \not=0$. Then combining (\ref{case1-2}) with (\ref{case1-3}) yields
\begin{eqnarray}
&& \int_D |\mcurl \vecu|^2 dx + \lambda \int_D N\vecu \cdot \overline{\vecu}\, dx = \frac{1}{\lambda} \int_D (N-I)\vecv \cdot \overline{\vecv} \,dx \label{positivecase1} \\
&=& \frac{1}{\lambda} \left(\int_D (N-I)\chi^2\vecv \cdot \overline{\vecv}\, dx + \int_D (N-I)(1-\chi^2)\vecv \cdot \overline{\vecv}\, dx\right) \nonumber \\
&=& \frac{1}{\lambda} \int_D (N-I)(1-\chi^2)\vecv \cdot \overline{\vecv} \, dx\left(1+ \frac{\int_D (N-I)\chi^2 \vecv \cdot \overline{\vecv} \, dx}{\int_D (N-I)(1-\chi^2)\vecv \cdot \overline{\vecv}\, dx}\right) \nonumber\\
&=& \frac{1}{\lambda}(n-1) \int_\mathcal{O} (1-\chi^2)|\vecv|^2 \,dx\left(1+ \frac{\int_D (N-I)\chi^2 \vecv \cdot \overline{\vecv} \,dx}{(n-1)\int_\mathcal{O}(1-\chi^2)|\vecv|^2 \,dx}\right) \label{negativecase1}
\end{eqnarray}
From Lemma \ref{lema1} we have for sufficiently large $\lambda$
\begin{equation*}
\frac{\left|\int_D (N-I)\chi^2 \vecv \cdot \overline{\vecv}\,dx\right|}{(1-n)\int_\mathcal{O}(1-\chi^2)|\vecv|^2 dx}<\frac{K(N_{max}+1)}{\lambda} <1,
\end{equation*}
where $N_{max}$ is supremum over $D$ of  the largest eigenvalue of $N$, which implies
$$
\Re\left(1+ \frac{\int_D (N-I)\chi^2 \vecv \cdot \overline{\vecv} dx}{(n-1)\int_\mathcal{O}(1-\chi^2)|\vecv|^2 dx}\right) >0.
$$
Then, since $n-1<0$, the real part of (\ref{negativecase1}) is non positive for sufficiently large $\lambda$ but the real part of (\ref{positivecase1}) is non negative. Hence the only possibility is $\vecu=0, \vecv=0$, i.e. $\vecE=\vecE_0=0$.

Let us next consider $n-1>0$ in $\mathcal{O}$, and let $\vecu=\vecE-\vecE_0$, $\vecv=\lambda \vecE$. Then
\begin{eqnarray}
\mcurl \mcurl \vecu+\lambda \vecu = -(N-I)\vecv \quad &\mbox{in}& \quad D  \label{2maxwell1} \\
\mcurl \mcurl \vecv+\lambda N \vecv = 0 \quad &\mbox{in}& \quad D \label{2maxwell} \\
\nu \times \vecu = \nu \times (\mcurl \vecu) =0 \quad &\mbox{on}& \quad \Gamma \label{2dirichlet}
\end{eqnarray}
Using the  same argument as for (\ref{case1-1})
\begin{equation} \label{case2-1}
\int_D \mcurl^2 \overline{\vecu} \cdot \vecv dx+ \lambda \int_D N \vecv \cdot \overline{\vecu} dx =0
\end{equation}
Multiplying (\ref{2maxwell1}) by $\overline{\vecv}$ yields
$$
\int_D  \overline{\vecv} \cdot \mcurl^2 \vecu dx + \lambda \int_D \overline{\vecv} \cdot \vecu dx + \int_D (N-I)\vecv \cdot \overline{\vecv} dx=0
$$
Combining the conjugate of the  above with (\ref{case2-1}) yields
\begin{equation} \label{case2-2}
\lambda \int_D (N-I)\overline{\vecu} \cdot \vecv dx = \int_D \overline{N-I} \overline{\vecv} \cdot \vecv dx
\end{equation}
Multiplying (\ref{2maxwell1}) by $\overline{\vecu}$ and integrating by parts yields
\begin{eqnarray*}
\int_D |\mcurl \vecu|^2 dx + \lambda \int_D |\vecu|^2 dx + \int_D (N-I)\vecv \cdot \overline{\vecu} dx=0.
\end{eqnarray*}
Note that since $N$ is symmetric, then $(N-I)\overline{\vecu} \cdot \vecv=(N-I)\vecv \cdot \overline{\vecu}$ and hence
\begin{equation}\label{case2-3}
\int_D |\mcurl \vecu|^2 dx + \lambda \int_D |\vecu|^2 dx + \int_D (N-I) \overline{\vecu}\cdot \vecv dx=0.
\end{equation}
Then combining (\ref{case2-2}) with (\ref{case2-3}) yields
\begin{eqnarray}
&& \int_D |\mcurl \vecu|^2 \,dx + \lambda \int_D |\vecu|^2 \,dx = -\frac{1}{\lambda} \int_D \overline{N-I} \,\overline{\vecv} \cdot \vecv \,dx  \label{positivecase2} \\
&=& -\frac{1}{\lambda} \left(\int_D \chi^2 \overline{N-I}\, \overline{\vecv} \cdot \vecv \,dx + \int_D (1-\chi^2)\overline{N-I}\, \overline{\vecv} \cdot \vecv \,dx\right) \nonumber \\
&=& -\frac{1}{\lambda} \int_D (1-\chi^2) \overline{N-I}\, \overline{\vecv} \cdot \vecv\, dx\left(1+ \frac{\int_D \chi^2 \overline{N-I}\, \overline{\vecv} \cdot \vecv dx}{\int_D (1-\chi^2)\overline{N-I} \,\overline{\vecv} \cdot \vecv dx}\right) \nonumber\\
&=& -\frac{1}{\lambda} \int_\mathcal{O} (n-1)(1-\chi^2)|\vecv|^2 dx\left(1+ \frac{\int_D \chi^2 \overline{N-I}\, \overline{\vecv} \cdot \vecv \,dx}{(n-1)\int_\mathcal{O}(1-\chi^2)|\vecv|^2 \,dx}\right). \label{negativecase2}
\end{eqnarray}
From Lemma \ref{lema1} we have for sufficiently large $\lambda$
\begin{equation*}
\frac{\left|\int_D \chi^2 \overline{N-I} \, \overline{\vecv} \cdot \vecv \, dx\right|}{(n-1)\int_\mathcal{O}(1-\chi^2)|\vecv|^2\,dx}<\frac{K(N_{max}+1)}{\lambda} <1.
\end{equation*}
Then  
$$
\Re\left(1+ \frac{\int_D \chi^2 \overline{N-I}\, \overline{\vecv} \cdot \vecv \, dx}{\int_\mathcal{O} (n-1)(1-\chi^2)|\vecv|^2\, dx}\right) >0.
$$
Therefore, since $n-1>0$, the real part of (\ref{negativecase2}) is non positive for sufficiently large $\lambda$ but the real part of (\ref{positivecase2}) is non negative. Hence the only possibility is $\vecu=0, \vecv=0$, i.e. $\vecE=\vecE_0=0$.
\end{proof} \proofend

\section{Discreteness of transmission eigenvalues}
Recall that in Section \ref{piece},  we have proved that $\tLk$ is  a Fredholm operator. Hence to show discreteness we will use the analytic Fredholm  theory \cite{coltonkress}. To this end we must show that there exists $k$ such that $\tLk$ is injective.
\begin{lemma} \label{injec}
Assume that $N$ satisfies \ref{massum}. There exists a purely imaginary $k$ with sufficiently large $|k|>0$ such that $\tilde{\bf L}(k)$ is injective.
\end{lemma}
\begin{proof}
Let us extend $N$ to $\R^3 \backslash \overline{D}$ by $N=nI$ where $n$ is the constant $N|_{\mathcal O}$. Assume there exists $\left( \begin{array}{c} \vecM \\ \vecJ \end{array} \right)$ such that $\Lk \left( \begin{array}{c} \vecM \\ \vecJ \end{array} \right)=0$. We will show that if $k$ is purely imaginary with large modulus, then $\left( \begin{array}{c} \vecM \\ \vecJ \end{array} \right)=0$. Recalling (\ref{integralA}), we define
\begin{align*}
\left( \begin{array}{c} \vecM' \\ \vecJ' \end{array} \right)
=
\mathcal{A}^{-1}(k) \vecL^{\Gamma, \Sigma}(k)
\left( \begin{array}{c} \vecM \\ \vecJ \end{array} \right)
\end{align*}
and
\begin{eqnarray*}
\vecE_0(x) &=& \mcurl \int_{\Gamma} \vecM ({\bf y}) \Phi_k(x,y) ds_y +\int_\Gamma \vecJ({\bf y}) \Phi_k(\cdot,{\bf y})ds_y \nonumber \\ 
&+&  \frac{1}{k^2}\nabla \int_\Gamma \mdiv_T  \vecJ ({\bf y}) \Phi_k(\cdot, {\bf y}) ds_y \quad \mbox{in} \quad \R^3 \backslash \Gamma.
\end{eqnarray*}

From the definition of $\left( \begin{array}{c} \vecM' \\ \vecJ' \end{array} \right)$ there exists $\vecE \in \vecL(D_1)$, $D_1:=D\setminus \overline{\mathcal O}$, such that
\begin{eqnarray*}
\mcurl \mcurl \vecE-k^2 N\vecE = 0 \quad &\mbox{in}& \quad D_1  \\ \,
[\vecE \times \nu]^+ = \vecM' \quad &\mbox{on}& \quad \Sigma  \\  \,
[\mcurl \vecE \times \nu]^+ = \vecJ' \quad &\mbox{on}& \quad \Sigma.  
\end{eqnarray*}
Also we define 
\begin{eqnarray*}
\vecE(x) &=& \mcurl \int_{\Gamma} \vecM({\bf y}) \Phi_{k_2}(x,y) ds_y +\int_\Gamma \vecJ({\bf y}) \Phi_{k_2}(\cdot,{\bf y})ds_y \nonumber \\ 
&+&  \frac{1}{k_2^2}\nabla \int_\Gamma \mdiv_T  \vecJ({\bf y}) \Phi_{k_2}(\cdot, {\bf y}) ds_y \\
&-&\mcurl \int_{\Sigma} \vecM' ({\bf y}) \Phi_{k_2}(x,y) ds_y -\int_\Sigma \vecJ'({\bf y}) \Phi_{k_2}(\cdot,{\bf y})ds_y \nonumber \\ 
&-&  \frac{1}{k_2^2}\nabla \int_\Sigma \mdiv_T  \vecJ'({\bf y}) \Phi_{k_2}(\cdot, {\bf y}) ds_y \quad \mbox{in} \quad \R^3 \backslash ( \overline{D}_1\cup \Gamma ).
\end{eqnarray*}
Jump relations across $\Gamma$  applied to $\vecE, \vecE_0$ along with the  equation (\ref{int-n})  yield
\begin{eqnarray}
\mcurl \mcurl \vecE-k^2 N  \vecE = 0 \quad &\mbox{in}& \quad \R^3 \backslash \Gamma  \label{maxwell1-r3} \\
\mcurl \mcurl \vecE_0-k^2 \vecE_0 = 0 \quad &\mbox{in}& \quad \R^3 \backslash \Gamma \label{maxwell-r3} \\
(\nu \times \vecE)^\pm  = (\nu \times \vecE_0)^\pm \quad &\mbox{on}& \quad \Gamma \label{dirichlet-r3} \\ 
(\nu \times \mcurl \vecE)^\pm  = (\nu \times \mcurl \vecE_0)^\pm  \quad &\mbox{on}& \quad \Gamma.  \label{neumann-r3}
\end{eqnarray}
From Theorem \ref{nonITE} if $k$ is purely imaginary with large enough modulus then (\ref{maxwell1-r3})-(\ref{neumann-r3}) in $D$ only has trivial solutions. Since $N=nI$ where $n$ is a constant in $\R^3\backslash \overline{D}$, then the variational formulation of (\ref{maxwell1-r3})-(\ref{neumann-r3}) in $\R^3\backslash \overline{D}$ is (\ref{coercive_varitional}) where the right hand is $0$ and $\R^3 \backslash \Gamma$ is replaced by $\R^3\backslash \overline{D}$. Then $\vecU=0$ and hence $\vecE=0$, $\vecE_0=0$ in $\R^3 \backslash \Gamma$. The jump relations (\ref{jumprelations1})-(\ref{jumprelations4}) yield $\vecM=0$ and $\vecJ=0$ and this proves the lemma.
\end{proof} \proofend

Finally, combining Lemma \ref{fredholm-n} and Lemma \ref{injec}, we can immediately prove our main theorem using the analytic Fredholm theory \cite{coltonkress}.
\begin{theorem}
Assume that $N$ satisfies Assumption \ref{massum}. Then the set of the transmission eigenvalues in ${\mathbb C}$ is discrete.
\end{theorem}

\section*{\normalsize Acknowledgments}
The  research of F. Cakoni is supported in part by the  Air Force Office of Scientific Research  Grant  FA9550-13-1-0199 and NSF Grant DMS-1515072. The research of S. Meng  is supported in part by the Chateaubriand STEM fellowship. S. Meng greatly acknowledges the hospitality of the DeFI Team at Ecole Polytechnique.

\end{document}